\documentclass[11pt]{elsarticle}

\usepackage{lineno,hyperref}

\usepackage{float}
\usepackage{graphicx}
\usepackage{subfig}
\usepackage{mathrsfs}
\usepackage{graphics}
\usepackage{float}
\usepackage{amsmath}
\usepackage{amsfonts}
\usepackage{mathrsfs}
\usepackage{amsmath}

\usepackage{optidef}

\usepackage{optidef}
\usepackage{graphicx}
\usepackage{float}
\usepackage{graphicx}
\usepackage{subfig}
\usepackage{mathrsfs}
\usepackage{float}
\usepackage{amsmath}
\usepackage{amsfonts}
\usepackage{booktabs}
\usepackage{natbib}
\usepackage{bbm}
\usepackage{algorithm}
\usepackage[noend]{algpseudocode}
\usepackage{multirow}
\algdef{SE}[DOWHILE]{Do}{doWhile}{\algorithmicdo}[1]{\algorithmicwhile\ #1}

\usepackage[letterpaper, portrait, margin=1in]{geometry}

\usepackage{booktabs}
\usepackage{tabularx}

\usepackage{xcolor}
\hypersetup{
    colorlinks,
    linkcolor={red!50!red},
    citecolor={blue!50!blue},
    urlcolor={blue!80!blue}
}

\usepackage{optidef}

\newtheorem{theorem}{Theorem}
\newtheorem{theorem2}{Theorem2}
\newtheorem{theorem3}{Theorem3}
\usepackage{amsthm}
\newtheorem{prop}{Proposition}
\theoremstyle{definition}

\newtheorem{lemma}[theorem3]{Lemma}
\newtheorem{remark}[theorem]{Remark}
\newtheorem{example}[theorem2]{Example}

\usepackage{newtxtext} % use times new roman font style

\DeclareMathOperator*{\argmin}{arg\,min}

\makeatletter
\newcommand*{\centerfloat}{%
  \parindent \z@
  \leftskip \z@ \@plus 1fil \@minus \marginparwidth
  \rightskip \leftskip
  \parfillskip \z@skip}
\makeatother

\usepackage{amssymb,amsmath,caption}
\usepackage[hang,flushmargin]{footmisc}
\newcommand{\algorithmfootnote}[2][\footnotesize]{
  \let\old@algocf@finish\@algocf@finish% Store algorithm finish macro
  \def\@algocf@finish{\old@algocf@finish% Update finish macro to insert "footnote"
    \leavevmode\rlap{\begin{minipage}{\linewidth}
    #1#2
    \end{minipage}}%
  }%
}

\usepackage{setspace}
\DeclarePairedDelimiter\abs{\lvert}{\rvert}
\newcommand{\norm}[1]{\left\lVert#1\right\rVert}

\usepackage{natbib}
\usepackage{algorithm}
\usepackage[noend]{algpseudocode}
\algdef{SE}[DOWHILE]{Do}{doWhile}{\algorithmicdo}[1]{\algorithmicwhile\ #1}

\usepackage{appendix}

\usepackage{array}
\usepackage{threeparttable}
\usepackage{multirow}
\usepackage{rotating}
\usepackage{makecell}

\journal{}

\biboptions{authoryear}

\begin{document}
\begin{frontmatter}

\title{Robust Path Recommendations During Public Transit Disruptions Under Demand Uncertainty}
% \title{Appendix: Methodology for probabilistic Inference}


\author[label1]{Baichuan Mo\corref{mycorrespondingauthor}}
\author[label3]{Haris N. Koutsopoulos}
\author[label5]{Zuo-Jun Max Shen}
\author[label4]{Jinhua Zhao}
\address[label1]{Department of Civil and Environmental Engineering, Massachusetts Institute of Technology, Cambridge, MA 02139}
\address[label3]{Department of Civil and Environmental Engineering, Northeastern University, Boston, MA 02115}
\address[label5]{Department of Industrial Engineering and Operations Research, University of California, Berkeley, Berkeley, CA 94720}
\address[label4]{Department of Urban Studies and Planning, Massachusetts Institute of Technology, Cambridge, MA 20139}

\cortext[mycorrespondingauthor]{Corresponding author}

\begin{abstract}
When there are significant service disruptions in public transit systems, passengers usually need guidance to find alternative paths. This paper proposes a path recommendation model to mitigate congestion during public transit disruptions. Passengers with different origins, destinations, and departure times are recommended with different paths such that the system travel time is minimized.  We model the path recommendation problem as an optimal flow problem with uncertain demand information. To tackle the lack of analytical formulation of travel times due to capacity constraints, we propose a simulation-based first-order approximation to transform the original problem into a linear program. Uncertainties in demand are modeled using robust optimization to protect the path recommendation strategies against inaccurate estimates. A real-world rail disruption scenario in the Chicago Transit Authority (CTA) system is used as a case study. Results show that even without considering uncertainty, the nominal model can reduce the system travel time by 9.1\% (compared to the status quo), and outperforms the benchmark capacity-based path recommendation.  The average travel time of passengers in the incident line (i.e., passengers receiving recommendations) is reduced more (-20.6\% compared to the status quo). After incorporating the demand uncertainty, the robust model can further reduce system travel times. The best robust model can decrease the average travel time of incident-line passengers by 2.91\% compared to the nominal model. The improvement of robust models is more prominent when the actual demand pattern is close to the worst-case demand. 
\end{abstract}

\begin{keyword}
Path recommendation; Robust optimization; Rail disruptions; Demand uncertainty
%%\MSC[2017] 00-01\sep  99-00
\end{keyword}

\end{frontmatter}

% \linenumbers % LINE NUMBER

\section{Introduction}\label{intro}
\subsection{Background}
Public transit (PT) systems play an important role in urban mobility. However, with aging systems, continuous expansion, and near-capacity operations, service disruptions often occur. These incidents may result in delays, cancellation of trips, and economic losses \citep{cox2011transportation}.

This study considers significant service disruptions in public transit systems where the service (or line/route) is interrupted for a relatively long period of time (e.g., 1 hour). During a disruption, affected passengers need to find an alternative path or use other travel modes (such as transfer to another bus route). However, due to a lack of knowledge of the system state (especially during incident time), the alternative routes chosen by passengers may not be optimal or even cause more congestion \citep{mo2022impact}. For example, during a rail disruption, most of the passengers may choose bus routes that are parallel to the interrupted rail line as an alternative. However, given the limited capacity of buses, the parallel bus line may be over-saturated and passengers have to wait for a long time to board due to being denied boarding (or left behind).  

\subsection{Objectives and Challenges}
One of the strategies to better guide passengers is to provide path recommendations so that the passenger flows are re-distributed in a better way and the system travel times are reduced. This can be seen as solving an \textbf{optimal passenger flow distribution (or assignment) problem} over a public transit network. However, there are several challenges to this problem. 
\begin{itemize}
    \item First, as the objective is to reduce the system travel time, an analytical formulation to calculate passengers' travel times is needed. However, a passenger's waiting times at the boarding and transfer stations are not only determined by other waiting passengers but also those who already boarded the same line as they reduce the vehicle's capacity \citep{de1993transit}. This complicated interaction makes it difficult to have an analytical formulation for passengers' travel time when the left behind is not negligible (which is usually the case during service disruptions). More details on this challenge are elaborated in Section \ref{sec_travel_time_liter}.
    \item Second, there are many uncertainties in the system, such as the number of passengers using the PT system during incidents (i.e., demand uncertainty), incident duration, and whether passengers would follow the recommendations or not (i.e., behavior uncertainty). Previous studies have not considered uncertainties in modeling an optimal passenger flow problem.
\end{itemize}

This study aims to propose a path recommendation model to reduce crowding during public transit disruptions, also taking into account uncertainties due to inaccurate demand estimates. Different from previous recommendation systems that focus on maximizing individual preferences, this study targets a system objective by minimizing the total travel time of all passengers (including those who are not in the incident line/area). To address the aforementioned first challenge, we propose a simulation-based linearization to convert the total system travel time to a linear function of path flows using a first-order approximation, which leads to a tractable optimization problem. For the second challenge, this study focuses on the demand uncertainty (i.e., how many passengers will use the PT system during a service disruption) and models it within the robust optimization (RO) framework. The proposed approach is applied in a case study using data from the Chicago Transit Authority (CTA) system during a real-world urban rail disruption. 
% Results show that even without considering uncertainty, the nominal model can reduce the system travel time by 9.1\% (compared to the status quo), and outperforms the benchmark (capacity-based) recommendations. After incorporating the uncertainties, the robust model can further reduce the system travel time. 

The main contributions of this paper are as follows:
\begin{itemize}
    % \item To the best of the authors' knowledge, this is the first article dealing with path recommendations under public transit service disruptions with a system-level objective. Most of the previous studies on path recommendation are conducted at individual or a single OD level. That is, the main objective is to find available (or shortest) routes or maximize individual preference given an OD pair when the network is interrupted by incidents, which is more like designing an incident-aware trip planner. 
    \item To tackle the non-analytical system travel time calculation, we propose a simulation-based linearization to convert the total system travel time to a linear function of path flows using first-order approximation. Importantly, we utilize the physical interaction between passengers and vehicles in a public transit system to efficiently calculate the gradient (i.e., marginal change of travel time) without running the simulation multiple times (as opposed to traditional black-box optimization). 
    \item We use RO to model the demand uncertainty which protects the model against inaccurate demand estimation. Specifically, we derive the closed-form robust counterpart with respect to the intersection of one ellipsoidal and three polyhedral uncertainty sets. These uncertainties capture the demand variations and the potential demand reduction during an incident. We also provide a feasible way of combining historical and survey data to quantify the uncertainty parameters.
\end{itemize}

The remainder of this paper is organized as follows. The literature review is presented in Section \ref{sec_liter}. In Section \ref{method}, we describe the problem and discuss the solution methods.  Section \ref{sec_model_exten} discusses model extensions and generalizability. We apply the proposed framework to the CTA system as a case study in Section \ref{sec_case_study}. The model results are analyzed in Section \ref{sec_results}. Finally, we conclude the paper and summarize the main findings in Section \ref{sec_conclusion}.

\section{Literature review}\label{sec_liter}
\subsection{Supply-side incident management}
During a disruption, transit operators usually need to adjust services such as re-schedule timetables, re-route services, or design shuttle buses. \citet{jespersen2009disruption} mention that the disruption management process often involves solving three interrelated problems sequentially: timetable adjustment, rolling stock rescheduling, and crew rescheduling. These are supply-side incident management strategies that are different from path recommendations (demand side). Supply-side strategies are widely explored in the literature. For example, timetable rescheduling has been explored from both train-oriented \citep{d2008reordering, d2009advanced, corman2010tabu, corman2012bi, corman2014dispatching, louwerse2014adjusting, zhan2015real} and passenger-oriented \citep{schobel2007integer, schachtebeck2010wait, dollevoet2012delay, kroon2015rescheduling, gao2016rescheduling} aspects, where the former pays more attention to the details of the rail system and the handling of disruptions or disturbances, focusing on minimizing the delays of trains or the number of canceled trains. The latter aims at minimizing passengers' total delay after a disruption or disturbance. For shuttle bus designs, \citet{kepaptsoglou2009bus} propose a methodological framework for planning and designing an efficient bus bridging network. \citet{jin2016optimizing} use a column generation procedure to dynamically generate demand-responsive candidate bus routes for shuttle bus design. A more comprehensive review of supply-side recovery models and algorithms for real-time railway disturbance and disruption management can be found in \citet{cacchiani2014overview}.

\subsection{Path recommendations during incidents}
Most previous studies on path recommendations under incidents were conducted at a single OD level. That is, the main objective is to find available routes or the shortest path given an OD pair when the network is interrupted by incidents. For example, \citet{bruglieri2015real} designed a trip planner to find the fastest path in the public transit network during service disruptions based on real-time mobility information. \citet{bohmova2013robust} developed a routing algorithm in urban public transportation to find reliable journeys that are robust against system delays. \citet{roelofsen2018assessing} provided a framework for generating and assessing alternative routes in case of disruptions in urban public transport systems. To the best of the authors' knowledge, none of the previous studies have considered path recommendations at the system level, that is, providing path recommendations for passengers of different OD pairs and with different departure times so that the system travel time is reduced.

\subsection{Passenger evacuation under emergencies}
Providing path recommendations during disruptions is related to the topic of passenger evacuation under emergencies. The objective of evacuation is usually to minimize the total evacuation time. In general, these papers can be categorized into micro-level and macro-level based on how passenger flows are modeled and the spatial scope of the study area. 

The micro-level studies usually use an agent-based simulation model to evaluate different evacuation strategies within some infrastructure. For example, \citet{wang2013simulation} simulated passenger evacuation under a fire emergency in Metro stations.  \citet{chen2017modelling} developed four modeling approaches including a queuing model and an agent-based simulation to calculate the evacuation
time under different emergency situations and evacuation plans. \citet{hassannayebi2020hybrid} used an agent-based and discrete-event simulation model to assess the service level performance and crowdedness in a metro station under various disruption scenarios (e.g., train failure in the tunnel and fire at the station gallery). \citet{zhou2019optimization} proposed a hybrid bi-level model to optimize the number and initial locations of leaders who guide passengers' evacuation in urban rail transit stations during an evacuation.

The macro-level studies consider a larger study area (e.g., city-level) and aim to evacuate passengers from the incident area through various transportation modes. For example, \citet{abdelgawad2012large} developed an evacuation model to determine the routing and scheduling of subway and bus transit services used to alleviate congestion pressure during the evacuation of busy urban areas. \citet{wang2019optimization} proposed an optimal bus bridging design method under operational disruptions on a single metro line. \citet{tan2020evacuating} proposes an evacuation model with urban bus networks as alternatives in the case of common metro service disruptions by jointly designing the bus lines and frequencies. 

The macro-level passenger evacuation is similar to the setup of this study, but with the following major differences. First, in our paper, the service disruption is not as severe as an emergency situation. The service will recover after a period of time and passengers are allowed to wait at a station. They do not necessarily need to cancel trips or follow evacuation plans as required in evacuation studies. Second, in this study, we assume that the service adjustment is known. The focus is on providing information to passengers to better utilize the existing resources/capacities of the system (demand side). However, the evacuation studies, since usually assuming the whole system breaks down, mainly focus on designing new services, such as routing and re-scheduling (supply side).

\subsection{Travel time calculation in public transit networks}\label{sec_travel_time_liter}
Passengers' travel time has two components: in-vehicle time and waiting time. In-vehicle time is not affected by passenger flows once passengers are onboard, thus is easy to model (e.g., modeled as a constant). However, the waiting time is more complicated to calculate if the system is congested with left behind due to capacity constraints. 

Passengers' travel time is usually modeled in the context of transit assignment, using two major approaches: frequency-based (static) and schedule-based (dynamic). In the frequency-based transit assignment approach, the waiting time is either assumed to be inversely proportional to the (effective) service frequency \citep{wu1994transit, schmocker2011frequency, nielsen2000stochastic}, or modeled as a congestion function (e.g., BRP) of previously boarded flows and new arrival flows with exogenously-calibrated parameters \citep{de1993transit}. The former method does not consider the left behind, and the latter only outputs a generalized waiting cost (rather than the waiting time as the vehicle capacity is not explicitly modeled) and requires a dedicated calibration process. Therefore, the frequency-based transit assignment model is not suitable for this study because congestion and left behind are not negligible during disruptions.  
 
In terms of the schedule-based models \citep{nguyen2001modeling, hamdouch2008schedule, hamdouch2014new, schmocker2008quasi}, the waiting time can only be obtained after a dynamic network loading (or simulation) process. For example, \citet{schmocker2008quasi} used the fail-to-board probability to model the left behind. This probability is updated after each network loading and can be used to calculate the waiting time. However, in this way, the waiting time is still constant within each iteration. There is no direct way to formulate waiting time as a function of path flows. 

Since formulating travel time as a function of path flows remains a challenge, the optimal passenger flow distribution in transit networks has no closed-form formulation. This study proposes a simulation-based first-order approximation to solve the original problem iteratively. With the proposed tractable linear programming model, uncertainties can also be incorporated.  

\subsection{Passenger queuing in over-saturated scenarios}
Since the difficulty of travel time calculation arises from the waiting time due to being left behind, we also review previous studies on modeling passenger queuing in over-saturated scenarios. Passenger left behind is usually modeled by the following nonlinear constraint:
\begin{align}
    \text{Num boarding passengers} = \min \{\text{Num waiting passengers}, \text{Remaining capacity} \}
\end{align}
This constraint is resolved by the following methods in the literature: 1) transferring to a linear constraint with binary decision variables then solved by heuristics or other algorithms \citep{gao2016rescheduling, shi2021operations}, 2) meta-heuristics (e.g., genetic algorithm (GA), sequential quadratic programming) \citep{yang2012cooperative, wang2015passenger}, 3) iterative convex programming \citep{wang2015efficient}, 4) approximate dynamic programming \citep{yin2016energy, shi2021operations}, 5) effective passenger loading time period (with binary decision variables) \citep{niu2013optimizing}. Among these methods, modeling with binary decision variables requires solving large-scale integer programming (the number of decision variables equal to the number of time intervals times the number of platforms). This is usually solved by some heuristics and may not be applicable in a large-scale network. Another category of meta-heuristics methods (like GA) is not efficient and does not well utilize domain-specific knowledge. Specifically, iterative convex programming is slightly similar to our method. In each iteration, the approach fixes the number of waiting passengers and onboard passengers based on the timetable from the last iteration and a simulation model. In our study, we also have a fixed “flow pattern” from the last iteration, but we also capture the “marginal change” in flows using a first-order approximation (see Section \ref{sub_foa} for details).   

\subsection{Simulation-based optimization}
Simulation-based optimization methods are designed to solve optimization problems where the objective function and its derivatives are difficult and expensive to evaluate. These methods have been widely used to solve the problems of congestion pricing \citep{chen2016time, he2017optimal}, traffic signal control \citep{osorio2013simulation, osorio2015urban, osorio2015energy, chong2018simulation}, transit scheduling \citep{zhang2017simulation}, route choice estimation \citep{mo2021calibrating, mo2022ex}, ride-sharing \citep{cardin2017real}, supply chain management \citep{noordhoek2018simulation}, liner shipping \citep{dong2009container} and more. In general, there are three classes of methods for the SBO, including the direct search method, the gradient-based method, and the response surface (meta-model) method \citep{osorio2013simulation}. In this paper, the proposed simulation-based first-order approximation is similar to a combination of the gradient-based and response surface (meta-model) methods. Specifically, we use the first-order approximation as a meta-model for the original objective function. In terms of the gradient calculation, instead of calling the simulation multiple times for the gradient evaluation (e.g., simultaneous perturbation stochastic approximation \citep{spall1997one}), we propose an efficient way to calculate the gradient based on its physical meaning. This greatly saves computational time compared to typical gradient-based methods. 

\subsection{Robust optimization (RO)}

RO is a common approach to handling data uncertainty in optimization problems. RO generally needs to first specify a scope of some uncertain parameters. The scope is referred to as the ``uncertainty set''. The optimization problem is conducted over the worst-case realizations within the specified uncertainty set. 
This method is suitable for applications where there are uncertainties related to the model input parameters and when uncertainties can lead to significant penalties or infeasibility in practice. Since the solutions are optimal under the worst-case scenario, we treat the outputs of RO as a robust solution. 

The solution method for RO problems involves generating a deterministic equivalent formulation, called the robust counterpart. Computational tractability of the robust counterpart has been a major practical difficulty \citep{BEN:09}.
A variety of uncertainty sets have been identified for which the robust counterpart is reasonably tractable \citep{bertsimas2011theory}. 

The studies on RO have grown substantially over the past decades. Seminal papers include \citep{ben1998robust}, \citep{ben1999robust} and \citep{bertsimas2004price}. 
Comprehensive surveys on the early literature can be found in \citet{BEN:09} and \citet{bertsimas2011theory}. The development of the RO methodology has allowed researchers to tackle problems with data uncertainty in a range of fields. Examples include renewable energy network design \citep{xiong2016distributionally}, supply chain operations \citep{ma2018distribution}, health care logistics \citep{wang2019distributionally}, and ride-hailing \citep{guo2021robust}. 

However, to the best of the authors' knowledge, no existing papers have incorporated RO techniques into path recommendations during service disruptions.  This research gap is important to address given the potentially inaccurate estimates of demand in public transit networks during an incident.

\section{Methodology}\label{method}
% \vspace{3pt}

\subsection{Event-based public transit simulator}
Before introducing the path recommendation, we first describe an event-based public transit simulator that is used across this study \citep{mo2020capacity}, especially for simulation-based linearization.  

\subsubsection{Simulator design}
Figure \ref{fig_model} summarizes the main structure of the simulator. The inputs for the simulator are time-dependent OD demand (or smart card data), path shares, network structure, and train movement data (or timetable). Three objects are defined: trains, queues, and passengers. Trains are characterized by routes, train ID, current locations, and capacities. Passengers are queued based on their arrival times. Three different types of passengers are represented: left-behind passengers who were denied boarding from previous trains, new tap-in passengers from outside the system, and new transfer passengers from other lines. The left-behind passengers are usually at the head of the queue. 

An event-based modeling framework is used to load the passengers onto the network. Two types of events are considered: train arrivals and train departures. The events are sorted by time and processed sequentially until all events are successfully completed during the analysis period. Train event lists (arrivals and departures) are generated according to the actual train movement data or timetable. Each event contains a train ID, occurrence time, and location (platform). Passengers are assigned to a path based on the corresponding input path shares. Note that in this study, a ``path'' is defined with specific boarding and transfer stations and lines. We assume passengers following a path will only board vehicles belonging to the specific line, even though there are multiple lines that serve a trip segment. Hence, there is no ``common line'' problem \citep{de1993transit} in this study because ``common lines'' will be treated as different paths. 

\begin{figure}[htb] 
\centering
\includegraphics[width= 0.7\linewidth]{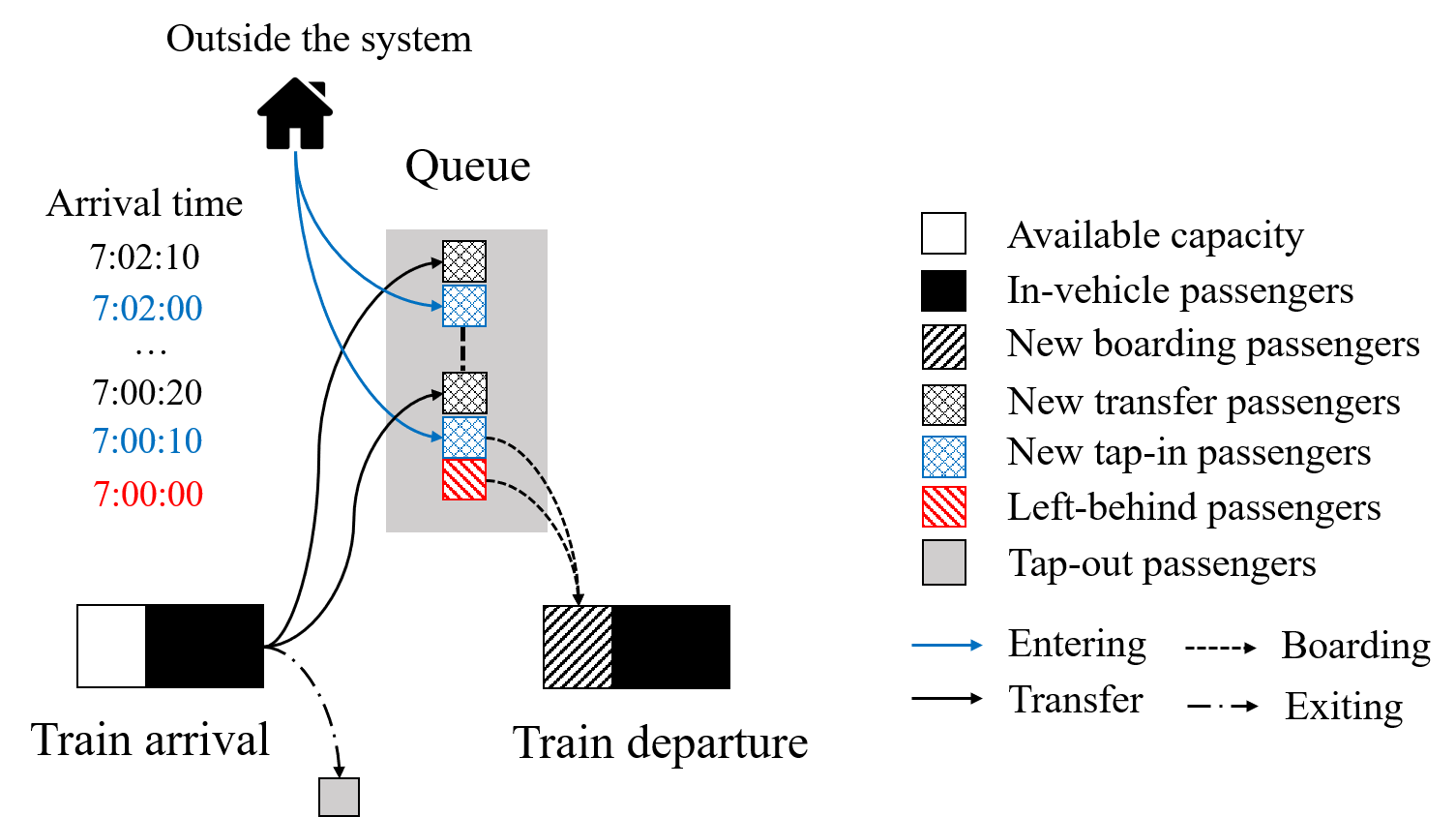}
\caption{Structure of the network loading model (adapted from \citet{mo2020capacity}}
\label{fig_model}
\end{figure}

For an arrival event, the train offloads passengers who reach their destination or need to transfer at the station and updates its state (e.g. train load and in-vehicle passengers). For passengers who reach their destinations, their tap-out times are calculated by adding their egress time. For those who transfer at the station, their arrival times at the next platform are calculated based on the transfer time. The transfer passengers are added to the waiting queue in order of their arrival times at the next platform. 

For departure events, the queue on the platform is updated by the new tap-in passengers, that is, passengers who arrive at the platform after the last train departed are added to the queue based on their arrival times. Passengers board the train according to a First-Come-First-Serve (FCFS) discipline until the train reaches its capacity. Passengers who cannot board are left behind and wait in the queue for the next train. The states of the train and the waiting queue are updated accordingly.

The simulator can record every passenger's trajectory during the whole travel process, including tap-in time, platform arrival time, boarding time, alighting time, tap-out time, etc. 

\subsubsection{Simulating service disruptions}
Given a service disruption, the event list is modified to incorporate the incident's impact on the supply side. Specifically, all incidents' impacts can be reflected by changes in vehicles' arrival and departure times. For example, the blockage of a rail line can be represented by some vehicles in the line having long dwell times at the corresponding stations during the incident period. The dispatching of shuttle buses can be seen as adding a new set of events (vehicle arrivals and departures) associated with the new bridging route. The headway adjustment of existing routes can also be captured by the new vehicle arrival and departure times. In this way, the event-based simulator can conveniently model service disruptions without changing the framework. It is worth noting that using the change of timetable to capture the incident impact on supply is also applicable to multi-platform scenarios (i.e., a platform serving different lines or different types of train capacities). Different from typical re-scheduling problems where the design of the new timetable needs to consider the train conflicts in the multi-platform scenario,  in this study, the timetable is given, where the possible conflicts are already considered in the new timetable. In addition, the timetable change can also capture the ``partially blocked'' platform. Details are illustrated in \ref{sec_sim_supply_change}.

From the passenger side, when an incident happens, all passengers in blocked trains are offloaded to the nearest platform. Depending on the input path choices (i.e., recommendation strategies), offloading passengers are re-assigned to a new alternative path and join the queues at the corresponding boarding station. After reassigning the offloading passengers, the simulator continues to run from the incident time to the end of the simulation period (note that passengers who have not entered the system when the incident occurs will have a new path choice depending on the input path choices). 

\subsection{Problem description}
Consider a service disruption in an urban rail system starting at time $T_s$ and ending at $T_e$. During the disruption, some stations in the incident line (or the whole line) are blocked. Passengers in the blocked trains are usually offloaded to the nearest platforms. To respond to the incident, some changes in the operations are made, such as dispatching shuttle buses, rerouting existing services, short-turning in the incident line, headway adjustment, etc. Assume that we have all information about the operating changes\footnote{That is, we assume during the disruption, operators would first change the supply to accommodate for the disruption, then provide path recommendations that incorporate the supply changes.}. These changes define a new PT service network and alternative path sets. Our objective is to design an origin-destination (OD) based recommendation system. That is, when the incident happens, passengers can use their phones, websites, or electrical boards at stations to access the recommendation system. They input their \textbf{origin station, destination station, and departure time} to get a recommended path. The recommendation aims to minimize the system travel time, that is, the sum of all passengers' travel times, including passengers at nearby lines or bus routes without incidents (note that these passengers may experience additional crowding due to transfer passengers from the incident line).

Let $\mathcal{K}$ be the predetermined set of all OD pairs that may need path recommendations. $\mathcal{K}$ is defined based on whether an OD pair is affected by the incident or not. Operators usually need a period called ``response time'' (e.g., 10 to 20 minutes) to generate the service changes. Let the response time be $\eta$. We assume that the path recommendations start at $T_s + \eta$. Note that the origins for passengers who are already in the system at time $T_s + \eta$ (e.g., offloaded passengers from the blocked vehicles) is their current locations (as opposed to their initial origins such as the boarding stations). We aim to provide recommendations for passengers whose OD pairs are in $\mathcal{K}$ and departure times are in the range from $T_s + \eta$ to some time point after $T_e$, since the congestion may last longer than $T_e$ and passengers departing after $T_e$ may also need guidance. Suppose that the period of recommendation starts at a time point ($h_0$) and consists of time intervals ($h_1,...,h_H$) of equal length $\tau$ (e.g., 10 minutes). Specifically, $h_0$ represents the time point at $T_s + \eta$. Recommendations at $T_s + \eta$ focus on passengers who are offloaded from blocked vehicles or arrive between $T_s$ and $T_s+\eta$ (their departure times are $T_s + \eta$)\footnote{Note that some of those passengers may schedule their departure times after $T_s+\eta$. These passengers will be considered as demand in other time intervals}. And $h_t$ $(t\geq 1)$ represents the time interval $(T_s + \eta + (t-1)\tau, T_s + \eta + t\tau]$. Recommendations at $h_t$ $(t\geq 1)$ focus on passengers who were not in the system when the incident happened and their departure times are in ${(T_s +\eta + (t-1)\tau, T_s + \eta+ t\tau]}$ (or passengers who are in the system when the incident happens but scheduled to depart in ${(T_s +\eta+ (t-1)\tau, T_s +\eta+ t\tau]}$). Let the set of all recommendation times be $\mathcal{H} := \{h_0, h_1, ... , h_H\}$. It is worth noting that, the following description aims to solve the model at time point $h_0$ and generate path recommendations from $h_0$ to $h_H$. However, the methodology is easy to be extended to a rolling horizon implementation where the model can be solved at any given time interval $\tilde{h} \in \mathcal{H}$. In this way, the service operation and demand information can be updated over time. Details of this discussion can be found in Section \ref{sec_model_rolling}.

Given the new operations during the incident, we obtain a feasible path set $R_k$ for each OD pair $k$. Note that $R_k$ includes all feasible services that are provided by the PT operator. A path $r \in R_k$  may be waiting for the system to recover (i.e., using the incident line), or transfer to nearby bus lines, using shuttle services, etc. We do not consider non-PT modes, such as Uber or driving for the following reasons: 1) The study aims to design a path recommendation system used by PT operators to provide path alternative recommendations to all PT users. Considering non-PT modes needs the supply information of all other travel modes and even consider non-PT users (such as the impact of traffic congestion on drivers), which is beyond the scope of this study. Future research may consider a multi-modal path recommendation system. 2) Passengers using non-PT modes can be simply treated as demand reduction for the PT system. So their impact on the PT system is still captured. 

Let $d_{hk}$ be the number of passengers using the PT system with OD pair $k \in \mathcal{K}$ and departure time $h \in \mathcal{H}$. It can be treated as the normal demand minus the number of passengers leaving the PT system. As we do not have full information about future demand and the number of passengers leaving the system, $d_{hk}$ is an uncertainty variable that will be discussed in Section \ref{sec_demand_uncertainty}. Let $f_{hkr}$ be the number of passengers departing at time interval $h$ using OD pair $k$ and path $r \in R_k$. By definition:
\begin{align}
    \sum_{r\in R_k} f_{hkr} = d_{hk} \quad \forall h\in\mathcal{H}, k\in\mathcal{K}
\end{align}
Let $p_{hkr}$ be the corresponding path share of $f_{hkr}$ (i.e., $p_{hkr} = f_{hkr}/d_{hk}$ and $ \sum_{r\in R_k} p_{hkr} = 1$). For convenience of description, we define $\mathcal{F} := \{(h,k,r): \forall h \in \mathcal{H}, \forall k \in \mathcal{K}, r \in R_k\}$ as the set of all path indices. Then the optimal flow problem can be formulated as:   
\begin{mini!}|s|[2]                 
    {\boldsymbol{f}, \boldsymbol{p}}
    {Z(\boldsymbol{f}) = \text{Sum of all passengers' travel time} \label{eq_opt_flow1}}
    {\label{eq_opt_flow}}
    {}
    \addConstraint{\sum_{r\in R_{k}} p_{hkr} }{= 1}{\;\;\;\forall \; h\in \mathcal{H},k\in \mathcal{K}\label{eq_const_naive1}}
    \addConstraint{f_{hkr} }{= d_{hk} \cdot p_{hkr}}{\;\;\;\forall \; (h,k,r) \in \mathcal{F} \label{const_demand_path_share}}
    \addConstraint{f_{hkr} }{\geq 0}{\;\;\;\forall \; (h,k,r) \in \mathcal{F}}
    \addConstraint{0 \leq p_{hkr} \leq 1}{}{\;\;\;\forall \; (h,k,r) \in \mathcal{F}\label{eq_const_naive2}}
\end{mini!}
where $\boldsymbol{f} := (f_{hkr})_{h,k,r\in\mathcal{F}}$ and $\boldsymbol{p} := (p_{hkr})_{h,k,r \in\mathcal{F}}$. $Z(\boldsymbol{f})$ is the system travel time which has no analytical expression. It can only be obtained after each network loading or simulation process (see Section \ref{sec_travel_time_liter}). Note that using both $\boldsymbol{f}$ and $\boldsymbol{p}$ in the optimization problem is redundant, but it is useful for explaining the methodology.

If there is no uncertainty in the system, the optimal path shares ($p_{hkr}^*$) obtained from the solution of Eq. \ref{eq_opt_flow} are the recommendation proportions. That is, for all passengers with OD pair $k$ and departure time $h$, the system will recommend them to use path $r$ with probability $p_{hkr}^*$. However, Eq. \ref{eq_opt_flow} is a conceptual formulation, it cannot be solved directly because $Z(\boldsymbol{f})$ has no analytical expression. Moreover, given the uncertainties in demand, the final recommended path shares may not be $p_{hkr}^*$. In the following sections, we elaborate on how to solve the robust ``optimal flow problem'' with demand uncertainties.

\subsection{Simulation-based linearization of the objective function}\label{sub_foa}
In this section, we propose a simulation-based linearization for the non-analytical $Z(\boldsymbol{f})$ based on a first-order approximation. $Z(\boldsymbol{f})$ can be approximated as:
\begin{align}
\hat{Z}(\boldsymbol{f}) = Z(\Tilde{\boldsymbol{f}}) + (\boldsymbol{f} - \Tilde{\boldsymbol{f}})^T\frac{\partial Z({\boldsymbol{f}})}{\partial\boldsymbol{f}}|_{\boldsymbol{f} = \Tilde{\boldsymbol{f}}}
\end{align}
where $\hat{Z}(\boldsymbol{f})$ is the first-order approximation of ${Z}(\boldsymbol{f})$. $\Tilde{\boldsymbol{f}}$ is a reference flow for the first-order approximation. $Z(\Tilde{\boldsymbol{f}})$ is the system travel time estimated by simulation with $\Tilde{\boldsymbol{f}}$ as input. $\frac{\partial Z({\boldsymbol{f}})}{\partial\boldsymbol{f}} = (\frac{\partial Z({\boldsymbol{f}})}{\partial f_{hkr} })_{{h,k,r} \in \mathcal{F}}$ is the gradient vector of $Z(\boldsymbol{f})$. As $\Tilde{\boldsymbol{f}}$ and $Z(\Tilde{\boldsymbol{f}})$ are pre-determined, the only unknown part is $\frac{\partial Z({\boldsymbol{f}})}{\partial\boldsymbol{f}}|_{\boldsymbol{f} = \Tilde{\boldsymbol{f}}}$. Notice that $\frac{\partial Z({\boldsymbol{f}})}{\partial f_{hkr} }|_{\boldsymbol{f} = \Tilde{\boldsymbol{f}}}$ represents the change of system travel time caused by one unit of flow change in $f_{hkr}$. It can be approximated as:
\begin{align}
\frac{\partial Z({\boldsymbol{f}})}{\partial f_{hkr} }|_{\boldsymbol{f} = \Tilde{\boldsymbol{f}}} \approx \frac{Z({\Tilde{\boldsymbol{f}} + \boldsymbol{e}_{hkr}}) - Z({\boldsymbol{\Tilde{\boldsymbol{f}}}})}{1} 
\label{eq_num_app_gra}
\end{align}
where $\boldsymbol{e}_{hkr}$ represents a vector with only the $(h,k,r)$-th element being 1 and others zero. Eq. \ref{eq_num_app_gra} represents the numerical approximation of the gradient. Now we only need to calculate $Z({\Tilde{\boldsymbol{f}} + \boldsymbol{e}_{hkr}}) - Z({\boldsymbol{\Tilde{\boldsymbol{f}}}})$. A naive method to do that is to run a simulation with ${\Tilde{\boldsymbol{f}} + \boldsymbol{e}_{hkr}}$ as input. However, as running the simulation is time-consuming, this method is not efficient. Note that since we already run a simulation with $\Tilde{\boldsymbol{f}}$ as input, it is possible to directly calculate the marginal change due to the additional unit of flow (i.e., calculate the additional travel time increase to the system if one additional flow is added to $\Tilde{f}_{hkr}$). 

Consider an example journey of $\tilde{f}_{hkr}$ in Figure \ref{fig_marginal_cost}. Let $\mathcal{M}_{hkr}$ be the set of passengers composing the flow of $\tilde{f}_{hkr}$ (i.e., the green passengers in Figure \ref{fig_marginal_cost}). These passengers have origin station $a_1$ and destination station $a_7$, and the path includes a transfer from station $a_4$ to station $a_5$. Let the average travel time of $\tilde{f}_{hkr}$ be $T_{hkr}^{\text{A}}(\boldsymbol{\tilde{f}})$. Suppose that one more passenger is added to $\tilde{f}_{hkr}$. 

\begin{figure}[htb]
\centering
\includegraphics[width = 0.5\linewidth]{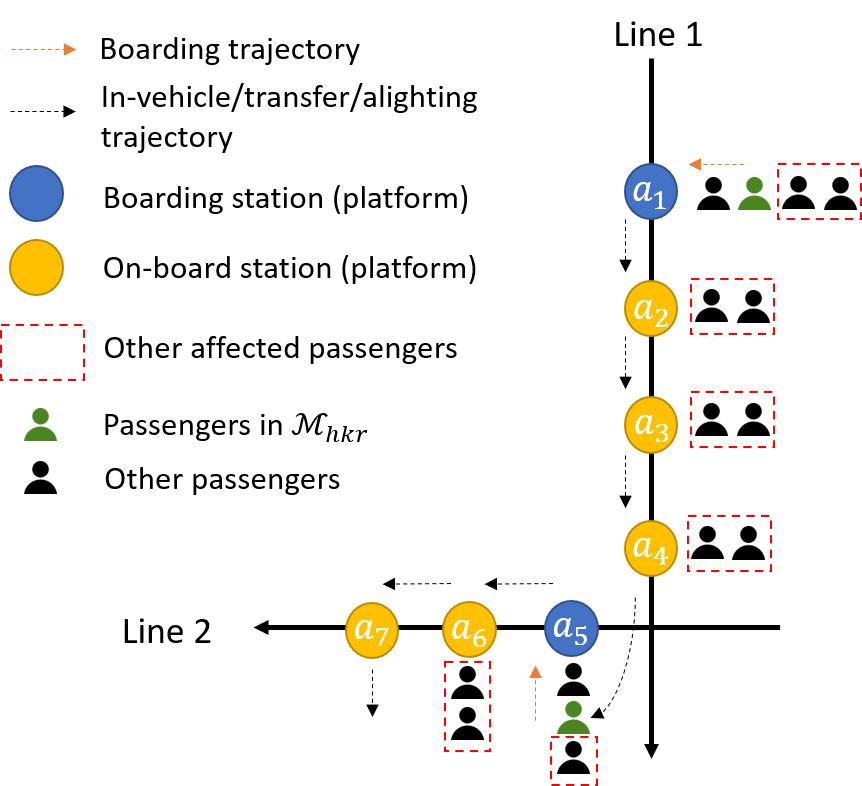}
\caption{Explanation for the impact of adding an additional one unit flow to the system}
\label{fig_marginal_cost}
\end{figure}

First of all, the system travel time is increased by $T_{hkr}^{\text{A}}(\boldsymbol{\tilde{f}})$ due to the increase in the flow amount. Note that considering the marginal calculation, we ignore the impact of the added passenger on the increase in $T_{hkr}^{\text{A}}(\boldsymbol{\tilde{f}})$. Besides, all passengers in the red-dashed square may experience higher travel times. Passengers at station $a_1$ and $a_5$ who queue behind the green passenger may have additional waiting time if the train that $\mathcal{M}_{hkr}$ used is full after departure (under the simulation results of $\boldsymbol{\tilde{f}}$), because the increase of the flow by one in $\tilde{f}_{hkr}$ will occupy one available capacity for these waiting passengers, and one of them will have to board the next train (i.e., wait for one more headway). Denote the total increase in system travel time for passengers queuing behind $\mathcal{M}_{hkr}$ as $T_{hkr}^{\text{Q}}(\boldsymbol{\tilde{f}})$. The detailed calculation of $T_{hkr}^{\text{Q}}(\boldsymbol{\tilde{f}})$ is shown in \ref{sec_sim_TQ}.

For passengers waiting at stations where $\mathcal{M}_{hkr}$ are already on-board (referred to as on-board stations, e.g., station $a_2$), adding one flow to $\tilde{f}_{hkr}$ reduces the available capacity when the vehicle arrives at these on-board stations. The queuing passengers at the onboard stations may not be able to board due to the reduction of capacity. Specifically, if a vehicle is full when it departs from an onboard station under flow pattern $\tilde{\boldsymbol{f}}$, adding one passenger to $\tilde{f}_{hkr}$ makes one passenger waiting at the on-board station unable to board his/her original boarded vehicle. And the system travel time is increased by one headway for each of these onboard stations. Denote the travel time increase for passengers waiting at on-board stations as $T_{hkr}^{\text{O}}(\boldsymbol{\tilde{f}})$. The detailed calculation of $T_{hkr}^{\text{O}}(\boldsymbol{\tilde{f}})$ is shown in \ref{sec_sim_TO}.

Therefore, in this way, depending on whether the vehicle is full or not under flow pattern $\tilde{\boldsymbol{f}}$, the increase in system travel time due to adding one passenger to $\tilde{f}_{hkr}$ can be calculated without running the simulation again. These increases come from three parts: 1) the average travel time of $\mathcal{M}_{hkr}$ due to increasing in flow amount (i.e., $T_{hkr}^{\text{A}}(\boldsymbol{\tilde{f}})$), 2) the additional waiting time for passengers queuing behind $\mathcal{M}_{hkr}$ (i.e., $T_{hkr}^{\text{Q}}(\boldsymbol{\tilde{f}})$), and 3) the additional waiting time for passengers queuing at $\mathcal{M}_{hkr}$'s on-board stations  (i.e., $T_{hkr}^{\text{O}}(\boldsymbol{\tilde{f}})$). Specifically, we have
\begin{align}
Z({\Tilde{\boldsymbol{f}} + \boldsymbol{e}_{hkr}}) - Z({\boldsymbol{\Tilde{\boldsymbol{f}}}}) = T_{hkr}^{\text{A}}(\boldsymbol{\tilde{f}}) + T_{hkr}^{\text{Q}}(\boldsymbol{\tilde{f}}) + T_{hkr}^{\text{O}}(\boldsymbol{\tilde{f}})
\end{align} 

Consequently, $\frac{\partial Z({\boldsymbol{f}})}{\partial\boldsymbol{f}}|_{\boldsymbol{f} = \Tilde{\boldsymbol{f}}}$ can be obtained from Eq. \ref{eq_num_app_gra}. Define $\boldsymbol{\beta}(\Tilde{\boldsymbol{f}}) := \frac{\partial Z({\boldsymbol{f}})}{\partial\boldsymbol{f}}|_{\boldsymbol{f} = \Tilde{\boldsymbol{f}}}$. Then the objective function becomes:
\begin{align}
    \hat{Z}(\boldsymbol{f})=  Z(\Tilde{\boldsymbol{f}}) + \boldsymbol{\beta}(\Tilde{\boldsymbol{f}})^T (\boldsymbol{f} - \Tilde{\boldsymbol{f}}) 
    \label{eq_obj_new}
\end{align}
where $\boldsymbol{\beta}(\Tilde{\boldsymbol{f}}) = (\beta_{hkr})_{h,k,r \in \mathcal{F}}$ and $\beta_{hkr} = \frac{\partial Z({\boldsymbol{f}})}{\partial f_{hkr} }|_{\boldsymbol{f} = \Tilde{\boldsymbol{f}}}$. Eq. \ref{eq_obj_new} is a linear function of $\boldsymbol{f}$, which supports for addressing uncertainties in the optimization problem.

% RC:

\subsection{Demand uncertainty}\label{sec_demand_uncertainty}

The uncertainty of $d_{hk}$ comes from two different parts. The first is the inherent demand variations across different days, and the second is the uncertainty in how many passengers leave the PT system during the incident. In this section, these two uncertainties are considered as a whole by introducing an ellipsoidal uncertainty set and three polyhedral uncertainty sets.

From constraint \ref{const_demand_path_share}, we can substitute $f_{hkr} = d_{hk}\cdot p_{hkr}$ to the objective function and rewrite Eq. \ref{eq_obj_new} as:
\begin{align}
   \hat{Z}(\boldsymbol{f})= \hat{Z}(\boldsymbol{p}) = Z(\Tilde{\boldsymbol{f}}) + \sum_{(h,k,r) \in \mathcal{F}}\beta_{hkr} \cdot (d_{hk}\cdot p_{hkr} - \tilde{f}_{hkr}) 
\end{align}
% Objective function: ($\beta$ is a function of $\Tilde{\boldsymbol{f}}$,  for simplicity we ignore it in the derivation)
Note that $\beta_{hkr}$ is a function of $\Tilde{\boldsymbol{f}}$, for simplicity we ignore $\Tilde{\boldsymbol{f}}$ in the derivation process. 

To model the uncertainty of $d_{hk}$, we introduce an auxiliary decision variable $t$ and rewrite the optimal flow problem as:
\begin{mini!}|s|[2]                 
    {\boldsymbol{p},t}
    {t \label{eq_obj_sub1_info}}
    {\label{eq_sub1_info}}
    {}
    \addConstraint{ t \geq Z(\Tilde{\boldsymbol{f}}) + \sum_{(h,k,r) \in \mathcal{F}}\beta_{hkr} \cdot (d_{hk}\cdot p_{hkr} - \tilde{f}_{hkr}) \label{cons_t}}{}{}
    \addConstraint{\text{Constraints (\ref{eq_const_naive1}) and (\ref{eq_const_naive2})} }{}{}
\end{mini!}
Constraint \ref{cons_t} can be rewritten as
\begin{align}
\sum_{h,k}\sum_{r\in R_k}\beta_{hkr} \cdot {d}_{hk} \cdot p_{hkr}  \leq t - Z(\Tilde{\boldsymbol{f}}) +  \sum_{(h,k,r)\in \mathcal{F}}\beta_{hkr}\tilde{f}_{hkr} 
\label{cons_t2}
\end{align}
Eq. \ref{cons_t2} can be written in a matrix form as:
\begin{align}
\boldsymbol{a}^T \boldsymbol{p} \leq b
\label{cons_t3}
\end{align}
where $\boldsymbol{a} \in \mathbb{R}^{|\mathcal{F}|}$ with the entry $a_{hkr} = \beta_{hkr} d_{hk}, \; \forall \;(h,k,r) \in \mathcal{F}$. And $b = t - Z(\Tilde{\boldsymbol{f}}) + \sum_{(h,k,r)\in \mathcal{F}}\beta_{hkr}\tilde{f}_{hkr}$. Define $\boldsymbol{d} = (d_{hk})_{h\in\mathcal{H},k\in\mathcal{K}}$.

\begin{prop}\label{prop_demand_uncertain}
If $\boldsymbol{d}$ is normally distributed with $\boldsymbol{d} \sim \mathcal{N}(\bar{\boldsymbol{d}}, \boldsymbol{\Sigma})$, then in a RO problem where constraint \ref{cons_t3} is guaranteed to be satisfied with probability of at least $1 - \varepsilon$ (i.e., $\mathbb{P}[\boldsymbol{a}^T \boldsymbol{p} \leq b] \geq 1 - \varepsilon$), the robust constraint can be formulated as:
\begin{align}
     (\boldsymbol{A} \bar{\boldsymbol{d}} + \boldsymbol{ADz})^T\boldsymbol{p} \leq b, \quad \forall \boldsymbol{z} \in \mathcal{Z}_{\text{E}}
     \label{const_z}
\end{align}
where $\boldsymbol{A} \in \mathbb{R}^{|\mathcal{F}|\times HK}$ with entry $A_{hkr,h'k'} = \beta_{hkr}$ if $h=h'$ and $k = k'$, otherwise $A_{hkr,h'k'} =0$. $\boldsymbol{D}$ is the Cholesky decomposition of $\boldsymbol{\Sigma}$ (i.e., $\boldsymbol{\Sigma }= \boldsymbol{D}  \boldsymbol{D}^T$). $\boldsymbol{z}$ are the perturbation variables (i.e., $\boldsymbol{d} =\bar{\boldsymbol{d}} + \boldsymbol{Dz}$) and $\mathcal{Z}_{\text{E}} = \left\{\boldsymbol{z} \in \mathbb{R}^{HK}: \norm{z}_2 \leq \rho_{1-\varepsilon} \right\}$ (i.e., the ellipsoidal uncertainty set). $\rho_{1-\varepsilon}$ is the $(1 - \varepsilon$)-percentile of a standard normal distribution. 
\end{prop}

\begin{proof}
\text{ }

\textbf{Step 1:} We first prove that $\mathbb{P}[\boldsymbol{a}^T \boldsymbol{p} \leq b] \geq 1 - \varepsilon$ is equivalent to $(\boldsymbol{A} \bar{\boldsymbol{d}})^T \boldsymbol{p} + \rho_{1-\varepsilon}  \norm{(\boldsymbol{AD})^T\boldsymbol{p}}_2 \leq b$.

Since $\boldsymbol{d}$ is normally distributed, we have $\boldsymbol{a} = \boldsymbol{A} \boldsymbol{d}$ is normally distributed with $\boldsymbol{a} \sim \mathcal{N}( \boldsymbol{A} \bar{\boldsymbol{d}}, \boldsymbol{A}\Sigma\boldsymbol{A}^T)$. Similarly,  $\boldsymbol{a}^T \boldsymbol{p} \in \mathbb{R}$ is also normally distributed with 
\begin{align}
\boldsymbol{a}^T \boldsymbol{p} \sim \mathcal{N}( (\boldsymbol{A} \bar{\boldsymbol{d}})^T \boldsymbol{p}, \boldsymbol{p}^T \boldsymbol{A} \boldsymbol{\Sigma }\boldsymbol{A}^T \boldsymbol{p})
\end{align}
If we want constraint \ref{cons_t3} to hold with probability at least $1 - \varepsilon$, it suffices to have:
\begin{align}
 (\boldsymbol{A} \bar{\boldsymbol{d}})^T \boldsymbol{p} + \rho_{1-\varepsilon} \sqrt{ \boldsymbol{p}^T \boldsymbol{A} \boldsymbol{\Sigma }\boldsymbol{A}^T \boldsymbol{p}} \leq b
 \label{cons_t4}
\end{align}
Substituting $\boldsymbol{\Sigma }= \boldsymbol{D}  \boldsymbol{D}^T$ into Eq. \ref{cons_t4} completes the proof of Step 1.

\textbf{Step 2:} We need to show that the robust counterpart of Eq. \ref{const_z} is $(\boldsymbol{A} \bar{\boldsymbol{d}})^T \boldsymbol{p} + \rho_{1-\varepsilon}  \norm{(\boldsymbol{AD})^T\boldsymbol{p}}_2 \leq b$.

Eq. \ref{const_z} is equivalent to:
\begin{align}
(\boldsymbol{A} \bar{\boldsymbol{d}})^T\boldsymbol{p} + \max_{\boldsymbol{z} \in \mathcal{Z}_{\text{E}}}(\boldsymbol{ADz})^T\boldsymbol{p} \leq b.
\label{eq_const_z_max}
\end{align}

Let $\delta(\boldsymbol{z}\mid\mathcal{Z}_{\text{E}})$ be the indicator function on set $\mathcal{Z}_{\text{E}}$:
\begin{align}
\delta(\boldsymbol{z}\mid\mathcal{Z}_{\text{E}})= 
\begin{cases}
    1,& \text{if } \boldsymbol{z} \in\mathcal{Z}_{\text{E}}\\
    0,              & \text{otherwise}
\end{cases}
\end{align}

Then the convex conjugate of $\delta(\boldsymbol{z} \mid \mathcal{Z}_{\text{E}})$ (also known as the \textbf{support function}) can be derived as \citep{Bertsimas2020}:
\begin{align}
  \delta^*(\boldsymbol{y}\mid\mathcal{Z}_{\text{E}}) = \sup_{\boldsymbol{z}\in \mathbb{R}^{HK}}\{\boldsymbol{y}^T \boldsymbol{z} - \delta(\boldsymbol{z}\mid\mathcal{Z}_{\text{E}})\}  = \sup_{ \boldsymbol{z}\in \mathcal{Z}_{\text{E}}} \boldsymbol{y}^T \boldsymbol{z} =  \rho_{1-\varepsilon} \norm{\boldsymbol{y}}_2
\end{align}
Therefore, Eq. \ref{eq_const_z_max} can be rewritten with the convex conjugate:
\begin{align}
    (\boldsymbol{A} \bar{\boldsymbol{d}})^T\boldsymbol{p} +  \delta^*( (\boldsymbol{AD})^T  \boldsymbol{p}\mid\mathcal{Z}) = (\boldsymbol{A} \bar{\boldsymbol{d}})^T\boldsymbol{p} + \rho_{1-\varepsilon}\norm{(\boldsymbol{AD})^T  \boldsymbol{p}}_2 \leq b
\end{align}
which finishes the proof of Step 2. Combining Steps 1 and 2 finishes the proof of the whole proposition.
\end{proof}

We observe that the ellipsoidal demand uncertainty performs like a regularization. It prevents $\boldsymbol{p}$ from being large in directions with considerable uncertainty in the demand. 

\begin{remark}\label{remark_1}
In the RO, the ellipsoidal uncertainty set can be used no matter what distribution $\boldsymbol{d}$ follows. If $\boldsymbol{d}$ is normally distributed, the parameter $\rho_{1 - \varepsilon}$ can be interpreted as the probability that constraint \ref{cons_t3} holds. The use of the multivariate normality assumption in Proposition \ref{prop_demand_uncertain} is for explaining the physical meaning of ellipsoidal uncertainty set and facilitating the choice of hyperparameters (i.e., $\rho_{1 - \varepsilon}$ and $\boldsymbol{D}$). Moreover, in the case study, we partially validate the multivariate normality assumption of $\boldsymbol{d}$ using smart card data. The Mardia’s Skewness Test \citep{cain2017univariate} shows that $\boldsymbol{d}$ has no significant skewness. 
\end{remark}

Eq. \ref{const_z} (i.e., the ellipsoidal uncertainty set) captures the correlation between demands at different time intervals and OD pairs. However, it does not impose any upper or lower bounds on $d_{hk}$. In reality, the demand level for a specific OD pair and time interval is usually bounded, which can be expressed as:
\begin{align}
d_{hk}^{\text{L}}  \leq d_{hk} \leq  d_{hk}^{\text{U}}
\label{eq_upper_lower_demand}
\end{align}
where $d_{hk}^{\text{L}}$ and $ d_{hk}^{\text{U}}$ are the corresponding lower and upper bounds for $d_{hk}$, respectively. Their values can be obtained from historical demand data. Eq. \ref{eq_upper_lower_demand} can be rewritten in a vector form as $\boldsymbol{d}^{\text{L}} \leq \boldsymbol{d} \leq \boldsymbol{d}^{\text{U}}$, where $\boldsymbol{d}^{\text{U}} = (d_{hk}^{\text{U}})_{h\in \mathcal{H},k\in \mathcal{K}}$ and $\boldsymbol{d}^{\text{L}} = (d_{hk}^{\text{L}})_{h\in \mathcal{H},k\in \mathcal{K}}$. Since we have $\boldsymbol{d} =\bar{\boldsymbol{d}} + \boldsymbol{Dz} $, a simple manipulation leads to
\begin{align}
\boldsymbol{d}^{\text{L}} - \bar{\boldsymbol{d}}  \leq  \boldsymbol{Dz} \leq  \boldsymbol{d}^{\text{U}} - \bar{\boldsymbol{d}}
\end{align}
We can rewrite it as a ``polyhedral uncertainty set'':  $\mathcal{Z}_{\text{P1}} = \left\{\boldsymbol{z} \in \mathbb{R}^{HK}: \boldsymbol{d}^{\text{L}} - \bar{\boldsymbol{d}}  \leq  \boldsymbol{Dz} \leq  \boldsymbol{d}^{\text{U}} - \bar{\boldsymbol{d}} \right\}$.

Eq. \ref{eq_upper_lower_demand} ensures the boundaries for each individual demand. Another similar constraint for the demand uncertainty is that: within a given time interval, the total demand across all OD pairs should also be bounded. This constraint can avoid some extreme scenarios that Eq. \ref{eq_upper_lower_demand} cannot capture (e.g., all $d_{hk}$ are at the lower or upper bounds). Mathematically:
\begin{align}
d_{h}^{\text{L}}  \leq \sum_{k\in\mathcal{K}}d_{hk} \leq  d_{h}^{\text{U}}
\label{eq_bound_total_demand}
\end{align}
where $d_{h}^{\text{L}}$ and $ d_{h}^{\text{U}}$ are the lower and upper bounds for the total demand in time interval $h$, which can be obtained from the historical demand.
Define $ \boldsymbol{S} \in  \mathbb{R}^{H \times HK}$, where the element $S_{h,h'k} = 1$ if $h = h'$, otherwise $S_{h,h'k} = 0$. Then Eq. \ref{eq_bound_total_demand} can be rewritten in a matrix form:
\begin{align}
\boldsymbol{d}_{\mathcal{H}}^{\text{L}} - \boldsymbol{S}\bar{\boldsymbol{d}} \leq  \boldsymbol{SDz} \leq \boldsymbol{d}_{\mathcal{H}}^{\text{U}} - \boldsymbol{S}\bar{\boldsymbol{d}} 
\label{eq_z_poly2}
\end{align}
where $\boldsymbol{d}_{\mathcal{H}}^{\text{U}} = (d_{h}^{\text{U}})_{h\in \mathcal{H}}$ and $\boldsymbol{d}_{\mathcal{H}}^{\text{L}} = (d_{h}^{\text{L}})_{h\in \mathcal{H}}$. And Eq. \ref{eq_z_poly2} can also be represented as a polyhedral uncertainty set:  $\mathcal{Z}_{\text{P2}} = \left\{\boldsymbol{z} \in \mathbb{R}^{HK}:\boldsymbol{d}_{\mathcal{H}}^{\text{L}} - \boldsymbol{S}\bar{\boldsymbol{d}} \leq  \boldsymbol{SDz} \leq \boldsymbol{d}_{\mathcal{H}}^{\text{U}} - \boldsymbol{S}\bar{\boldsymbol{d}} \right\}$.

As the RO aims to optimize under the ``worst case'' scenario and our objective function is the system travel time, intuitively, the worst-case scenario will be the largest demand in the uncertainty set. This may make the worst-case demand unrealistic since the extremely large demand seldom happens. What we expect in the RO is that the model can capture some critical OD pairs where the high demand in these OD pairs can make the system more congested (as opposed to high demand in all OD pairs). In order to let the RO capture critical OD pairs, we add an additional constraint on the total demand:
\begin{align}
\sum_{h \in\mathcal{H}, k\in\mathcal{K}}d_{hk} \leq  \Gamma \cdot \sum_{h \in\mathcal{H}, k\in\mathcal{K}}\bar{d}_{hk} 
\label{eq_z_poly3}
\end{align}
where $\Gamma > 0$ is a predetermined constant. $\Gamma = 1$ means we assume the total demand in the worst-case scenario is the same as the nominal one, but the spatial and temporal distributions are different. The worst-case scenario will have more demand on critical OD pairs but less demand on others. The value of $\Gamma$ can be determined based on the highest total demand observed over a time period.  

Similarly, Eq. \ref{eq_z_poly3} can be written in a matrix form:
\begin{align}
\boldsymbol{1}^T (\bar{\boldsymbol{d}} + \boldsymbol{Dz})\leq  \Gamma \cdot \boldsymbol{1}^T \bar{\boldsymbol{d}}
\label{eq_z_poly3_matrix}
\end{align}
where $\boldsymbol{1} \in \mathbb{R}^{HK}$ is a vector with all elements one. And we define another polyhedral uncertainty set:  $\mathcal{Z}_{\text{P3}} = \left\{\boldsymbol{z} \in \mathbb{R}^{HK}:\boldsymbol{1}^T (\bar{\boldsymbol{d}} + \boldsymbol{Dz})\leq  \Gamma \cdot \boldsymbol{1}^T \bar{\boldsymbol{d}} \right\}$.

Therefore, the final robust constraint for Eq. \ref{cons_t3} is
\begin{align}
     (\boldsymbol{A} \bar{\boldsymbol{d}} + \boldsymbol{ADz})^T\boldsymbol{p} \leq b, \quad \forall \boldsymbol{z} \in \mathcal{Z}_{\text{E}}\cap\mathcal{Z}_{\text{P}}\cap\mathcal{Z}_{\text{P2}}\cap\mathcal{Z}_{\text{P3}}
     \label{const_z_all}
\end{align}

To derive the robust counterpart of the constraint, we first introduce the following lemma. 

\begin{lemma}
\label{lemma_RC}
    \emph{For a constraint $\Bar{\boldsymbol{a}}^T \boldsymbol{x} + \delta^*(\boldsymbol{P}^T\boldsymbol{x} \mid \mathcal{Z}) \leq b$,
    let $\mathcal{Z}_1,...,\mathcal{Z}_k$ be closed convex sets, such that $\bigcap_i ri(\mathcal{Z}_i) \neq \emptyset$\footnote{$ri(\mathcal{Z}_i)$ indicates the relative interior of the set $\mathcal{Z}_i$.}, and let $\mathcal{Z} = \cap_{i=1}^k \mathcal{Z}_i$. Then,
    $$\delta^*(\boldsymbol{y}\mid \mathcal{Z}) =\min_{\boldsymbol{y}_1,...,\boldsymbol{y}_k} \{ \sum_{i=1}^k \delta^*(\boldsymbol{y}_i \mid \mathcal{Z}_i) \mid \sum_{i=1}^k \boldsymbol{y}_i = \boldsymbol{y} \}, $$
    and the constraint becomes
    $$\begin{cases} 
      \Bar{\boldsymbol{a}}^T \boldsymbol{x} + \sum_{i=1}^k \delta^*(\boldsymbol{y}_i \mid \mathcal{Z}_i) \leq b \\
      \sum_{i=1}^k \boldsymbol{y}_i = \boldsymbol{P}^T \boldsymbol{x}  
   \end{cases}$$
   where $\delta^*( \cdot \mid \cdot)$ is the support function (i.e., convex conjugate of the indicator function). 
   }
\end{lemma}

The proof of Lemma \ref{lemma_RC} can be found in \citet{ben2015deriving}. From Proposition \ref{prop_demand_uncertain}, we have $\delta^*(\boldsymbol{y}\mid \mathcal{Z}_{\text{E}}) =  \rho_{1-\varepsilon} \norm{\boldsymbol{y}}_2$. For the polyhedral uncertainty set, consider a general form $\mathcal{Z}_{\text{P}} = \left\{\boldsymbol{z} : \boldsymbol{Hz} \leq \boldsymbol{c} \right\}$. And the support function for $\mathcal{Z}_{\text{P}}$ is
\begin{align}
      \delta^*(\boldsymbol{y}\mid\mathcal{Z}_{\text{P}}) = \max_{\boldsymbol{z}}\{\boldsymbol{y}^T \boldsymbol{z} \mid  \boldsymbol{Hz} \leq \boldsymbol{c}\}  = \min_{\boldsymbol{u}}\{\boldsymbol{c}^T \boldsymbol{u} \mid  \boldsymbol{H}^T \boldsymbol{u} = \boldsymbol{y},  \boldsymbol{u} \geq 0\}
      \label{eq_sup_poly}
\end{align}
where the second equality follows from linear programming duality. Eq. \ref{eq_sup_poly} can be used to derive the support function for $\mathcal{Z}_{\text{P1}}$, $\mathcal{Z}_{\text{P2}}$, and $\mathcal{Z}_{\text{P3}}$. For example, consider the robust counterpart for Eq. \ref{eq_z_poly3}, we have
\begin{align}
    \delta^*(\boldsymbol{y}_6\mid\mathcal{Z}_{\text{P3}}) = \min_{{u}_3}\{(\Gamma -1)\cdot(\boldsymbol{1}^T \bar{\boldsymbol{d}} )\cdot {u}_3 \mid  (\boldsymbol{1}^T\boldsymbol{D})^T {u}_3 = \boldsymbol{y}_6,  {u}_3 \geq 0\}
\end{align}
where $\boldsymbol{y}_6 \in \mathbb{R}^{HK}$ and ${u}_3 \in \mathbb{R}$ are decision variables in the RO model. Note that the subscripts for $\boldsymbol{y}$ and $u$ (i.e., 6 and 3) are used for the consistency in Eq. \ref{eq_all_rc}. 

Based on Lemma \ref{lemma_RC}, the robust counterpart for Eq. \ref{const_z_all} is
\begin{subequations}
\label{eq_all_rc}
\begin{align}
& (\boldsymbol{A} \bar{\boldsymbol{d}})^T\boldsymbol{p} +  \rho_{1-\varepsilon}\norm{\boldsymbol{y}_1}_2  +  (\boldsymbol{d}^{\text{U}} - \bar{\boldsymbol{d}})^T \boldsymbol{u}_1  +  (\bar{\boldsymbol{d}} - \boldsymbol{d}^{\text{L}})^T \boldsymbol{u}_2 + (\boldsymbol{d}^{\text{U}}_{\mathcal{H}} - \boldsymbol{S}\bar{\boldsymbol{d}})^T \boldsymbol{v}_1  +  (\boldsymbol{S}\bar{\boldsymbol{d}}-  \boldsymbol{d}^{\text{L}}_{\mathcal{H}})^T \boldsymbol{v}_2 \notag \\
& + (\Gamma -1)\cdot(\boldsymbol{1}^T \bar{\boldsymbol{d}} )\cdot {u}_3 \leq b \\
& \boldsymbol{D}^T\boldsymbol{u}_1 = \boldsymbol{y}_2\label{eq_RO1_const2}\\
& -\boldsymbol{D}^T\boldsymbol{u}_2 = \boldsymbol{y}_3\\
& (\boldsymbol{SD})^T\boldsymbol{v}_1 = \boldsymbol{y}_4\\
& -(\boldsymbol{SD})^T\boldsymbol{v}_2 = \boldsymbol{y}_5\\
& (\boldsymbol{1}^T\boldsymbol{D})^T {u}_3 = \boldsymbol{y}_6\\
& \sum_{i=1}^6 \boldsymbol{y}_i = (\boldsymbol{AD})^T \boldsymbol{p} \\
& \boldsymbol{u}_1,\boldsymbol{u}_2, \boldsymbol{v}_1,\boldsymbol{v}_2,{u}_3 \geq 0\label{eq_RO1_const3}
\end{align}
\end{subequations}

Hence, the RO problem can be formulated as
\begin{subequations}
\begin{align}\label{eq_all_opt}
 \quad \min_{\boldsymbol{p},\boldsymbol{u},\boldsymbol{v},\boldsymbol{y},t} \quad & t \\
    \text{s.t.} \quad &
    \sum_{(h,k,r) \in \mathcal{F}}\beta_{hkr} \cdot d_{hk} \cdot  p_{hkr} + \rho_{1-\varepsilon}\norm{\boldsymbol{y}_1}_2  +  (\boldsymbol{d}^{\text{U}} - \bar{\boldsymbol{d}})^T \boldsymbol{u}_1  +  (\bar{\boldsymbol{d}} - \boldsymbol{d}^{\text{L}})^T \boldsymbol{u}_2 + (\boldsymbol{d}^{\text{U}}_{\mathcal{H}}  - \boldsymbol{S}\bar{\boldsymbol{d}})^T \boldsymbol{v}_1   \notag \\ 
    &  
  +  (\boldsymbol{S}\bar{\boldsymbol{d}}-  \boldsymbol{d}^{\text{L}}_{\mathcal{H}})^T \boldsymbol{v}_2 + (\Gamma -1)\cdot(\boldsymbol{1}^T \bar{\boldsymbol{d}} )\cdot {u}_3+Z(\Tilde{\boldsymbol{f}}) -  \sum_{(h,k,r) \in \mathcal{F}}\beta_{hkr}\tilde{f}_{hkr} \leq t \label{eq_RO1_const1}\\
    & \text{Constraints } (\ref{eq_RO1_const2}) - (\ref{eq_RO1_const3}) \\
    & \text{Constraints } (\ref{eq_const_naive1}) \text{ and } (\ref{eq_const_naive2})
\end{align}
\label{eq_RO1}
\end{subequations}
By eliminating $t$ and inserting constraint \ref{eq_RO1_const1} in the objective function it becomes
\begin{align}
\hat{Z}(\boldsymbol{p},\boldsymbol{u},\boldsymbol{v},\boldsymbol{y})^{\text{RC}} &=  \sum_{(h,k,r) \in \mathcal{F}}\beta_{hkr} \cdot (d_{hk} \cdot  p_{hkr} -\tilde{f}_{hkr} ) + \rho_{1-\varepsilon}\norm{\boldsymbol{y}_1}_2  +  (\boldsymbol{d}^{\text{U}} - \bar{\boldsymbol{d}})^T \boldsymbol{u}_1  +  (\bar{\boldsymbol{d}} - \boldsymbol{d}^{\text{L}})^T \boldsymbol{u}_2 \notag \\ 
    & + (\boldsymbol{d}^{\text{U}}_{\mathcal{H}} 
   - \boldsymbol{S}\bar{\boldsymbol{d}})^T \boldsymbol{v}_1  
  +  (\boldsymbol{S}\bar{\boldsymbol{d}}-  \boldsymbol{d}^{\text{L}}_{\mathcal{H}})^T \boldsymbol{v}_2 + (\Gamma -1)\cdot(\boldsymbol{1}^T \bar{\boldsymbol{d}} )\cdot {u}_3 +Z(\Tilde{\boldsymbol{f}}) 
    \label{eq_obj_RC}
\end{align}
which yields a second-order cone programming (SOCP).

\subsection{Solution procedure}
After incorporating the demand uncertainty, the final robust counterpart (RC) of the optimal flow problem can be formulated as:
\begin{subequations}
\begin{align}
    [RC (\Tilde{\boldsymbol{f}})] \quad \min_{\boldsymbol{p},\boldsymbol{u},\boldsymbol{v},\boldsymbol{y}} \quad & \hat{Z}(\boldsymbol{p},\boldsymbol{u},\boldsymbol{v},\boldsymbol{y})^{\text{RC}} =  \sum_{(h,k,r) \in \mathcal{F}}\beta_{hkr}(\Tilde{\boldsymbol{f}}) \cdot (d_{hk} \cdot  p_{hkr} -\tilde{f}_{hkr} ) + \rho_{1-\varepsilon}\norm{\boldsymbol{y}_1}_2  +  (\boldsymbol{d}^{\text{U}} - \bar{\boldsymbol{d}})^T \boldsymbol{u}_1   \notag \\ 
    &
      +  (\bar{\boldsymbol{d}} - \boldsymbol{d}^{\text{L}})^T \boldsymbol{u}_2
   + (\boldsymbol{d}^{\text{U}}_{\mathcal{H}} 
   - \boldsymbol{S}\bar{\boldsymbol{d}})^T \boldsymbol{v}_1  
  +  (\boldsymbol{S}\bar{\boldsymbol{d}}-  \boldsymbol{d}^{\text{L}}_{\mathcal{H}})^T \boldsymbol{v}_2 + (\Gamma -1)\cdot(\boldsymbol{1}^T \bar{\boldsymbol{d}} )\cdot {u}_3
  +  Z(\Tilde{\boldsymbol{f}})  \\
    \text{s.t.} \quad
    & \text{Constraints } (\ref{eq_RO1_const2}) - (\ref{eq_RO1_const3}) \\
    & \sum_{r\in R_k} p_{hkr} = 1  \quad \forall h \in \mathcal{H}, k \in \mathcal{K} \\
    & 0 \leq p_{hkr} \leq 1 \quad \forall (h,k,r) \in \mathcal{F}
\end{align}
\label{eq_ro_all}
\end{subequations}
This SOCP can be efficiently solved by inner interior point methods that are embedded in many existing solvers.

However, due to the first-order approximation of $Z(\boldsymbol{f})$, $\beta_{hkr}(\Tilde{\boldsymbol{f}})$ needs to be updated once a new flow pattern is obtained. Hence, after obtaining $\boldsymbol{p}^*$ from the RC problem, the simulation should be run again to update $\beta_{hkr}(\Tilde{\boldsymbol{f}})$.  Before that, the corresponding worst-case demand (WD), which will be used as the new $\Tilde{\boldsymbol{f}}$, is needed. It can be obtained by solving the worst case $\boldsymbol{z} \in \mathcal{Z}_{\text{E}}\cap\mathcal{Z}_{\text{P1}}\cap\mathcal{Z}_{\text{P2}}\cap\mathcal{Z}_{\text{P3}}$:
\begin{subequations}
\begin{align}
    [WD(\boldsymbol{p}^*)] \quad \max_{\boldsymbol{z}} \quad & (\boldsymbol{ADz})^T\boldsymbol{p}^*  \\
    \text{s.t.} \quad
    & \norm{\boldsymbol{z}}_2 \leq \rho_{1-\varepsilon}\\
    & \boldsymbol{d}^{\text{L}} - \bar{\boldsymbol{d}}  \leq  \boldsymbol{Dz} \leq  \boldsymbol{d}^{\text{U}} - \bar{\boldsymbol{d}}\\
    & \boldsymbol{d}_{\mathcal{H}}^{\text{L}} - \boldsymbol{S}\bar{\boldsymbol{d}} \leq  \boldsymbol{SDz} \leq \boldsymbol{d}_{\mathcal{H}}^{\text{U}} - \boldsymbol{S}\bar{\boldsymbol{d}}\\
    & \boldsymbol{1}^T (\bar{\boldsymbol{d}} + \boldsymbol{Dz})\leq  \Gamma \cdot \boldsymbol{1}^T \bar{\boldsymbol{d}}
\end{align}
\label{eq_WD}
\end{subequations}
If the solution for Eq. \ref{eq_WD} is $\boldsymbol{z}^*$, the worse case demand is $\boldsymbol{d}^* = \bar{\boldsymbol{d}} + \boldsymbol{Dz}^*$. Next, we can update $\boldsymbol{\beta}(\Tilde{\boldsymbol{f}})$ and $ Z(\Tilde{\boldsymbol{f}})$ as
\begin{align}
 Z(\Tilde{\boldsymbol{f}}), \boldsymbol{\beta}(\Tilde{\boldsymbol{f}}) = \textsc{Sim-FOA}(\boldsymbol{d}^*, \boldsymbol{p}^*)
 \label{eq_sim_FOA}
\end{align}
where $\Tilde{\boldsymbol{f}}$ in Eq. \ref{eq_sim_FOA} indicates $\tilde{f}_{hkr}= d_{hk}^* \cdot p_{hkr}^*$. And $\textsc{Sim-FOA}(\cdot)$ is a pseudo function of simulation plus first-order approximation as described in Section \ref{sub_foa}. 

The RC, WD, and $\textsc{Sim-FOA}(\cdot)$ problems need to be solved iteratively. This can be treated as a fixed-point problem. A conventional way to solve a fixed-point problem is the method of successive averages (MSA). In the typical system optimal \textbf{traffic} assignment problem, the optimal flow pattern is reached when for every OD pair, the marginal costs of all paths for this OD pair are the same. This implies that, ideally, when the flow distribution is optimal, we should have $\beta_{hkr}(\Tilde{\boldsymbol{f}}) = \beta_{hkr'}(\Tilde{\boldsymbol{f}})$ for all $r, r' \in R_k  \setminus  R_k^{\text{NoFlow}}$, where $R_k^{\text{NoFlow}} = \{{r\in R_k \mid f_{hkr} = 0}\}$ is the path set with zero flows. This implies that at the system optimal assignment, the marginal cost (travel time) of every non-zero flow path is the same (i.e., one cannot decrease the system travel time by switching passengers from one path to another). 

However, in our study, this cannot be set as the convergence criterion because, in the dynamic \textbf{transit} assignment context, the cost function is not continuous due to left behind. Adding one more passenger to a path may lead to the system travel time increased by one or more headways. The following example illustrates that $\beta_{hkr}(\Tilde{\boldsymbol{f}})$ can be arbitrarily large, which may cause the criterion of $\beta_{hkr}(\Tilde{\boldsymbol{f}}) = \beta_{hkr'}(\Tilde{\boldsymbol{f}})$ never being satisfied.

\begin{example}\label{example_large_beta}
\emph{Consider a single direction bus line with $N$ stations (Figure \ref{fig_example_large}) and a fixed headway $W$. Assume every bus has a capacity of 1. There is one passenger waiting at each station except for the first station (i.e., there are $N-1$ waiting passengers). Now assume that one more passenger is added to station 1. Since the capacity of buses is 1, the newly added passenger will force all waiting passengers to be left behind one more time. Hence, the total added system travel time is $ (N-1) \times W$. In this scenario, the $\beta_{hkr}(\Tilde{\boldsymbol{f}})$ associated with the added passenger can be arbitrarily large depending on the number of stations $N$. 
}
\begin{figure}[H]
\centering
\includegraphics[width = 0.7 \linewidth]{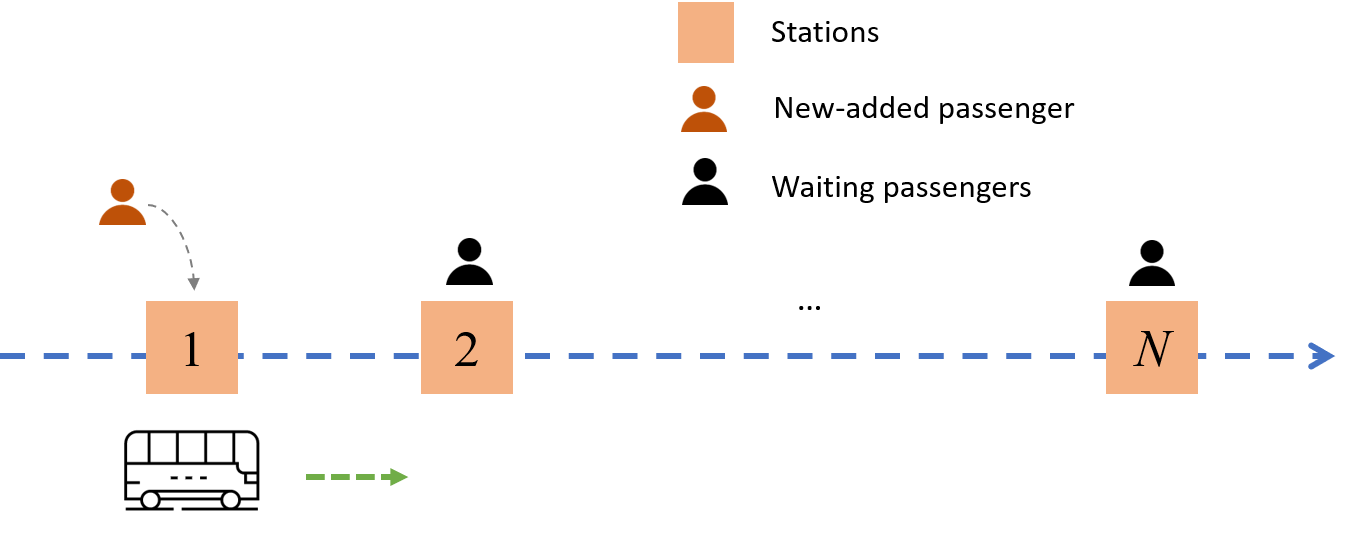}
\caption{Example for arbitrarily large $\beta_{hkr}(\Tilde{\boldsymbol{f}})$}
\label{fig_example_large}
\end{figure}
\end{example}

Therefore, in this study, we define the convergence criteria based on the value of system travel time (i.e., when the value of the system travel time is relatively stable within a range). Specifically, it is assumed that the MSA algorithm has converged if
% for the following reasons. First, constraint \ref{eq_RC_const_1} limits the values of path flow, the theoretical optimal flows (without constraints) may not be feasible. Second,
\begin{align}
  \abs*{Z(\Tilde{\boldsymbol{f}})^{(n)} - \frac{1}{N^{\text{Cvg}}}\sum_{n' = n - N^{\text{Cvg}}}^{n-1} Z(\Tilde{\boldsymbol{f}})^{(n')} } \leq \epsilon
 \label{eq_convergence}
\end{align}
where $Z(\Tilde{\boldsymbol{f}})^{(n)}$ is the system travel time at the $n$-th iteration and $\epsilon$ is a predetermined threshold. Eq. \ref{eq_convergence} means that when the current system travel time is close to its average value of the last $N^{\text{Cvg}}$ iterations, the algorithm terminates. Taking the average of the last $N^{\text{Cvg}}$ iterations can mitigate the impact of fluctuations caused by the discontinuity of the system travel time.  
 
% The marginal path travel time can be arbitrarily large. This can be illustrated by the following example. 

The whole solution algorithm is described in Algorithm \ref{alg_overall}.  Line 6 indicates the MSA step. Lines 10 and 11 mean that we will use the path shares with the smallest system travel time over the last $N^{\text{Cvg}} + 1$ iterations. 

\begin{algorithm} 
\caption{Solution procedure of the robust path recommendation} \label{alg_overall}
\begin{algorithmic}[1]
\State Initialize $\boldsymbol{p}^{(0)}$ (e.g., uniform path shares), $\boldsymbol{d}^{(0)}$ (e.g., nominal demand) and specify $N^{\text{Cvg}}, \epsilon$. 
\State Set iteration counter $n=0$.
\Do
    \State  $Z(\Tilde{\boldsymbol{f}})^{(n)}, \boldsymbol{\beta}(\Tilde{\boldsymbol{f}})^{(n)} = \textsc{Sim-FOA}(\boldsymbol{d}^{(n)}, \boldsymbol{p}^{(n)})$
    \State Solve the RC problem (Eq. \ref{eq_ro_all}) with $Z(\Tilde{\boldsymbol{f}})^{(n)}$ and $ \boldsymbol{\beta}(\Tilde{\boldsymbol{f}})^{(n)}$ as inputs, and return $\hat{\boldsymbol{p}}^{(n+1)}$
    \State ${\boldsymbol{p}}^{(n+1)} = \frac{1}{n+1}\hat{\boldsymbol{p}}^{(n+1)} + (1-\frac{1}{n+1}){\boldsymbol{p}}^{(n)} $
    \State Solve the WD problem (Eq. \ref{eq_WD}) with ${\boldsymbol{p}}^{(n+1)}$ as input and return $\boldsymbol{d}^{(n+1)}$
     \State $n = n + 1$
    % \State ${{\mu}^{i_{m} ,j_n}_r}^{(k)} = (1-\alpha^{(k)}){{\mu}^{i_{m} ,j_n}_r}^{(k-1)} + \alpha^{(k)}{{}{\tilde{\mu}^{i_{m} ,j_n}_r}}^{(k)}$
\doWhile{$n \leq N^{\text{Cvg}}$ or $\abs*{Z(\Tilde{\boldsymbol{f}})^{(n)} - \frac{1}{N^{\text{Cvg}}}\sum_{n' = n - N^{\text{Cvg}}}^{n-1} Z(\Tilde{\boldsymbol{f}})^{(n')} } > \epsilon$}
\State $n^* = \argmin_{n' = n-N^{\text{Cvg}},...,n} Z(\Tilde{\boldsymbol{f}})^{(n')}$
\State \Return $\boldsymbol{p}^{(n^*)}$
\end{algorithmic}
\end{algorithm}

Let $\boldsymbol{p}^*$ be the optimal path shares by from Algorithm \ref{alg_overall}. To realize the optimal path shares in the real world, the following system design can be used:
\begin{itemize} 
    \item Transit operators deploy the recommendation system to smartphone apps, websites, and electrical screens at stations.
    \item Passengers, when using the system, input their origins, destinations, and departure times.
    \item For a passenger input OD pair $k$ and departure time $h$, the system will return a single recommended path $r$ to them with probability $p_{hkr}^*$.
\end{itemize}
In this way, we expect the final path flows are close to the system optimal path flows if passengers follow the recommendation. In reality, passengers may have different preferences for different recommendations. That is, they may not follow the recommendations if they are provided with an unpreferred path. \ref{sec_path_pass} discusses how to solve another path-passenger matching problem that incorporates passengers' preferences.

\section{Model extensions}\label{sec_model_exten}
\subsection{Solving the model in a rolling horizon}\label{sec_model_rolling}
The model discussed in the previous section is a one-shot solution for path recommendation, which means the model will be run at the beginning of an incident ($h_0$) and output the recommendations for the whole period of interest $[h_0, h_H]$. In application, the model would be implemented in a rolling horizon framework. 

Specifically, at time interval $\tilde{h}$, we first update the demand and supply information, including new demand estimates, new demand uncertainty sets, new available path sets, new service routes and frequencies, new incident duration estimates, etc. Based on the formulation above (i.e., let $h_0 = \tilde{h}$)), we solve the model to obtain recommendations for time $[\tilde{h}, h_H]$. But we only implement the recommendation strategies for the current time $\tilde{h}$ (i.e., $p_{\tilde{h}kr}^*$). An illustration of the rolling horizon implementation is shown in Figure \ref{fig_rolling}. In this way, the new information obtained with the evolution of the incident and system operations can be used to improve model performance (this is known as adaptive RO). 

\begin{figure}[H]
\centering
\includegraphics[width = 0.7 \linewidth]{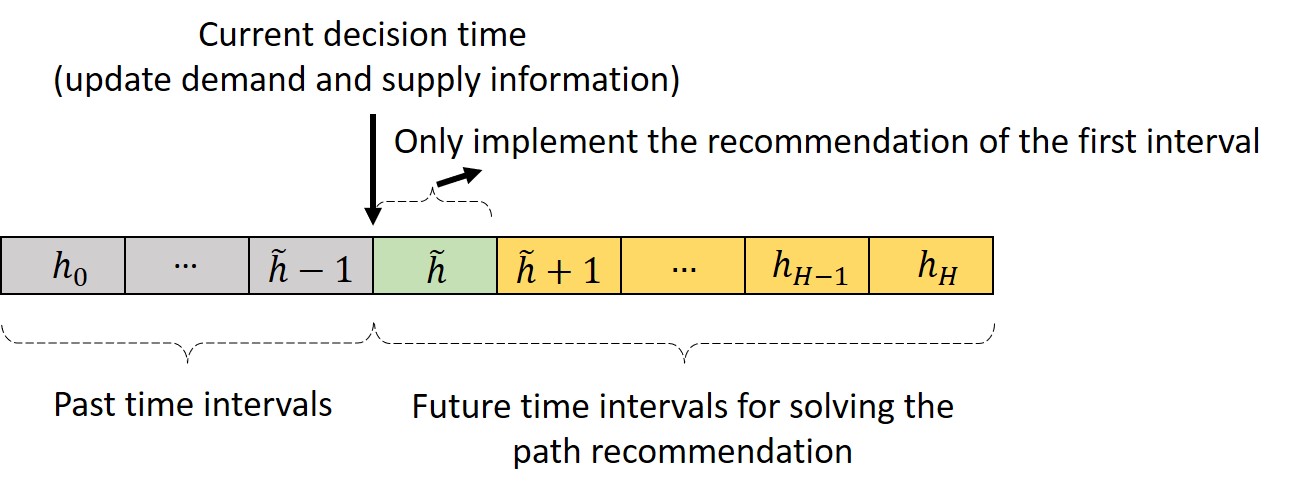}
\caption{Illustration of the rolling horizon implementation}
\label{fig_rolling}
\end{figure}

In the case study, we did not implement the rolling horizon for the following reasons. 1) First, as we use the CTA system as the case study, the operators in the selected incident did not update operation changes once the decisions have been made. Hence, it becomes less meaningful for the rolling horizon with fixed operations. 2) Second, there are computational time challenges in the case study. The main bottleneck comes from the simulation model. Running a simulation in the large-scale CTA network takes around 0.6 minutes. And our algorithm requires around 35 iterations to converge (see Section \ref{sec_model_converge}). 3) Third, as mentioned before, a more holistic implementation of the rolling horizon includes the update of the ``uncertainty set'', which implies an adaptive robust optimization. This requires additional derivations on how to use previous demand realizations to generate the new demand uncertainty set, which deserves separate future research.

\subsection{Incident duration uncertainty}\label{sec_dur_uncert}
In this study, we assume operators have a reasonable estimate of incident duration. However, it is possible that we can only obtain a distribution of incident duration. In this section, we show that our formulation can be easily extended to capture the incident duration uncertainty with stochastic optimization (SO) techniques\footnote{The reason for using SO, instead of RO, to capture incident duration uncertainty is that the worst-case scenario for the incident duration is always the largest one, which makes the problem trivial and may not reflect reality.}.  

Let the set of all possible incident scenarios be $\Omega$. For example, $\Omega$ may include incidents with a duration of 30, 40, or 50 minutes. For each scenario $\xi \in \Omega$, we denote $\beta_{hkr}(\Tilde{\boldsymbol{f}}; \xi)$ and $Z(\Tilde{\boldsymbol{f}}; \xi)$ as the approximated gradient and current system travel time under flow $\Tilde{\boldsymbol{f}}$ and incident scenario $\xi$. Hence, the objective function for the RO problem becomes:
\begin{align}
    \mathbb{E}[ \hat{Z}(\boldsymbol{p},\boldsymbol{u},\boldsymbol{v},\boldsymbol{y})^{\text{RC}}] & = \sum_{\xi \in \Omega} \mathbb{P}(\xi) \left[Z(\Tilde{\boldsymbol{f}};\xi) + \sum_{(h,k,r)\in\mathcal{F}}\beta_{hkr}(\Tilde{\boldsymbol{f}};\xi) \cdot (d_{hk}\cdot p_{hkr} - \tilde{f}_{hkr})\right]  + \rho_{1-\varepsilon}\norm{\boldsymbol{y}_1}_2  \notag \\ 
    & +  (\boldsymbol{d}^{\text{U}} - \bar{\boldsymbol{d}})^T \boldsymbol{u}_1  +  (\bar{\boldsymbol{d}} - \boldsymbol{d}^{\text{L}})^T \boldsymbol{u}_2 + (\boldsymbol{d}^{\text{U}}_{\mathcal{H}} 
   - \boldsymbol{S}\bar{\boldsymbol{d}})^T \boldsymbol{v}_1  
  +  (\boldsymbol{S}\bar{\boldsymbol{d}}-  \boldsymbol{d}^{\text{L}}_{\mathcal{H}})^T \boldsymbol{v}_2 + (\Gamma -1)\cdot(\boldsymbol{1}^T \bar{\boldsymbol{d}} )\cdot {u}_3
\end{align}
where $\mathbb{P}(\xi)$ is the probability of scenario $\xi$ being realized. The expectation above is taking over different incident scenarios. Define $Z(\Tilde{\boldsymbol{f}}; \Omega):= \sum_{\xi \in \Omega} \mathbb{P}(\xi) Z(\Tilde{\boldsymbol{f}};\xi)$ and $\beta_{hkr}(\Tilde{\boldsymbol{f}}; \Omega) := \sum_{\xi \in \Omega} \mathbb{P}(\xi)\beta_{hkr}(\Tilde{\boldsymbol{f}};\xi)$, substituting them into the objective function
\begin{align}
    \mathbb{E}[ \hat{Z}(\boldsymbol{p},\boldsymbol{u},\boldsymbol{v},\boldsymbol{y})^{\text{RC}}] =  \sum_{(h,k,r) \in \mathcal{F}}\beta_{hkr}(\Tilde{\boldsymbol{f}};\Omega) \cdot (d_{hk} \cdot  p_{hkr} -\tilde{f}_{hkr} ) + \rho_{1-\varepsilon}\norm{\boldsymbol{y}_1}_2  +  (\boldsymbol{d}^{\text{U}} - \bar{\boldsymbol{d}})^T \boldsymbol{u}_1  \notag \\ 
      +  (\bar{\boldsymbol{d}} - \boldsymbol{d}^{\text{L}})^T \boldsymbol{u}_2
  + (\boldsymbol{d}^{\text{U}}_{\mathcal{H}} 
  - \boldsymbol{S}\bar{\boldsymbol{d}})^T \boldsymbol{v}_1  
  +  (\boldsymbol{S}\bar{\boldsymbol{d}}-  \boldsymbol{d}^{\text{L}}_{\mathcal{H}})^T \boldsymbol{v}_2 +  (\Gamma -1)\cdot(\boldsymbol{1}^T \bar{\boldsymbol{d}} )\cdot {u}_3 + Z(\Tilde{\boldsymbol{f}};\Omega)  
\end{align}
As the constraints in the RO problem are not related to incident scenarios (i.e., $\beta_{hkr}(\Tilde{\boldsymbol{f}})$ and $Z(\Tilde{\boldsymbol{f}})$ are not included in the constraint part), this implies that incorporating the incident duration uncertainty with SO only requires a change in the objective function.

\section{Case study design}\label{sec_case_study}
In the case study, we consider an actual incident in the Blue line of the Chicago Transit Authority (CTA) urban rail system (Figure \ref{fig_incident}). The incident starts at 8:14 AM and ends at 9:13 AM on Feb 1st, 2019 due to infrastructure issues between Harlem and Jefferson Park stations. The entire Blue Line was suspended. During the disruption, the Loop (Chicago CBD area) is the destination for most passengers. Usually, there are four paths leading to the Loop: 1) using Blue Line (i.e., waiting for the system to recover), 2) using the parallel bus lines, 3) using the North-South (NS) bus lines to transfer to the Green Line, and 4) using the West-East (WE) bus lines to transfer to the Brown Line. Based on the service structure, we can construct the route sets $R_k$ for each OD pair $k$.

\begin{figure}[htb]
\centering
\includegraphics[width = 0.8 \linewidth]{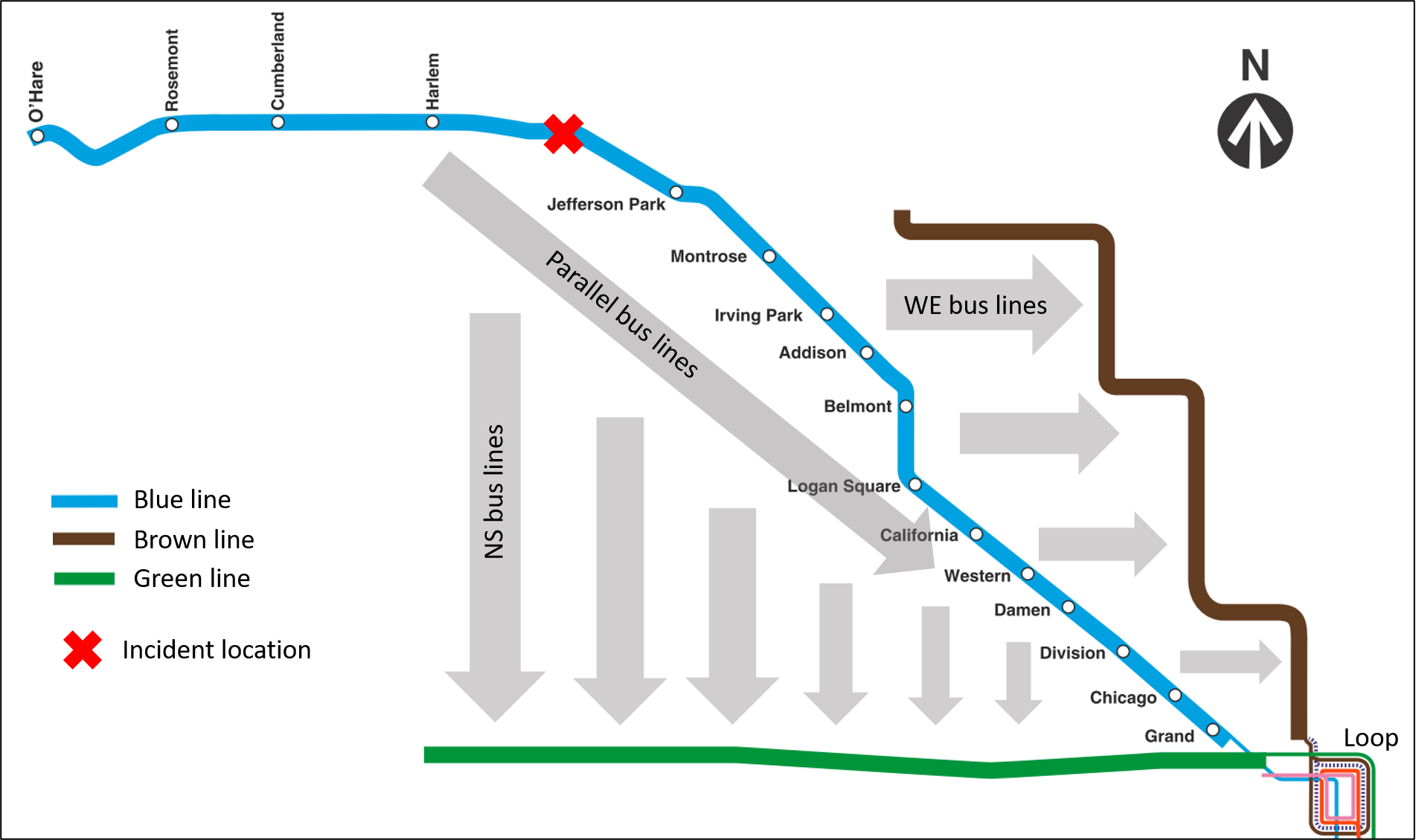}
\caption{Case study diagram}
\label{fig_incident}
\end{figure}

\subsection{Parameter setting}
$\mathcal{K}$ is the set of all OD pairs with origins at the Blue Line and destinations at the Loop. The response time is set as $\eta = 0$ for simplicity (i.e., assuming a quick response for operators). The time interval is set to $\tau = 10$ mins. The time period with recommendation is set as $h_H = 10$, corresponding to 9:44 - 9:54 AM (i.e., 50 minutes after the end of the incident).
In this study, we assume that the incident duration is known or can be reasonably estimated. The factor of total demand level $\Gamma$ is set to 1.1, which is the 90\% percentile of the total demand distribution.

The validation of the simulation model's performance is shown in \ref{sec_validate_sim}. Results show that the model can capture the passenger and vehicle interactions well in the CTA system.

\subsection{Quantification of uncertainty sets}\label{sec_quant_uncertainty_set}
The demand uncertainty is determined by the nominal demand $\bar{\boldsymbol{d}}$, covariance matrix $\boldsymbol{\Sigma}$ (which can be used to get $\boldsymbol{D}$), and upper and lower bounds for demand (i.e., $\boldsymbol{d}^{\text{U}}$, $\boldsymbol{d}^{\text{L}}$, $\boldsymbol{d}_{\mathcal{H}}^{\text{U}}$, $\boldsymbol{d}_{\mathcal{H}}^{\text{L}}$). These can be estimated from historical demand. However, as the demand on the incident day is smaller than usual given that some passengers may leave the system, we cannot directly use normal day smart card data as historical demand. One possible solution is to use data from previous days with similar incidents. Nevertheless, this is usually unavailable due to the lack of enough similar incidents. Hence, in this study, we first use survey results and historical smart card data to generate ``synthetic historical demand'' samples, and then estimate the uncertainty set from the samples.

There are two sources of demand uncertainty: 1) the inherent demand variations across different days and 2) the uncertainty of how many passengers left the PT system during the incident. The first part can be captured by historical smart card data (without incidents). The second part is approximated by the survey results. According to previous survey-based studies, the proportion of the passengers leaving the PT system during incidents is around 10\%$\sim$30\% \citep{lin2016transit,rahimi2020analysis}. Then, the ``synthetic historical demand'' is generated as follows:
\begin{itemize} 
    \item Collect smart card data from a recent workday and calculate the demand vector without passengers leaving the system for each $(h,k)$ (the demand for $h=0$, i.e., offloading passengers, can be obtained using the simulation model).
    \item For each $(h,k)$, we randomly draw a proportion of leaving passengers from a uniform distribution $\mathcal{U}(10\%, 30\%)$\footnote{We use uniform distribution because we have no distributional information of the leaving passenger proportions}. The demand after removing the leaving passengers is the incident period demand vector. 
\end{itemize}
We collected a total of 16 weekdays from Jan 2019 (the previous month of the incident day) and generated 16 sample demand vectors. The mean value is used as the nominal demand $\bar{\boldsymbol{d}}$ and the covariance matrix $\boldsymbol{\Sigma}$ is estimated from these samples. The upper and lower bounds for demand (i.e., $\boldsymbol{d}^{\text{U}}$, $\boldsymbol{d}^{\text{L}}$, $\boldsymbol{d}_{\mathcal{H}}^{\text{U}}$, $\boldsymbol{d}_{\mathcal{H}}^{\text{L}}$) are set as the samples' maximum and minimum values, respectively. 

The hyperparameter $\rho_{1 - \varepsilon}$ for the ellipsoidal uncertainty set are chosen from these values: $\{$0, 0.25, 0.52, 0.84, 1.28, 1.64, 2.33$\}$, which corresponds to the $\{$50, 60, 70, 80, 90, 95, 99$\}$ percentiles of the standard normal distribution. Note that $\rho_{1 - \varepsilon} = 0$ represents the case of no uncertainty (i.e., nominal model).

% The quantification of demand uncertainty sets does not use any information on the incident day and it is applicable for the real-world incident. Therefore, even if using synthetic historical demand may introduce errors, as long as we show that this approach of uncertainty set definition yields reasonably good results, it is acceptable for real-world applications. 

\subsection{Data description}

The nominal and actual (incident day) demand comparison is shown in Figure \ref{fig_compare}. The total nominal demand is 5,499, similar to the total actual demand (5,531), implying that introducing a proportion of leaving passengers (i.e., 10\% - 30\%) can capture the demand reduction on the incident day.  
We also observe that the aggregate nominal demand for each time interval is similar to that of the incident day. The major differences happen at the first two time intervals ($h=0,1$). However, looking at the demand for each $(h,k)$ (Figure \ref{fig_demand_hk}), the differences are more prominent. The discrepancy between nominal and actual demands indicates the potential for the RO approach to perform better. 

\begin{figure}[H]
\centering
\subfloat[Total demand for each time interval $h$]{\includegraphics[width=0.5\textwidth]{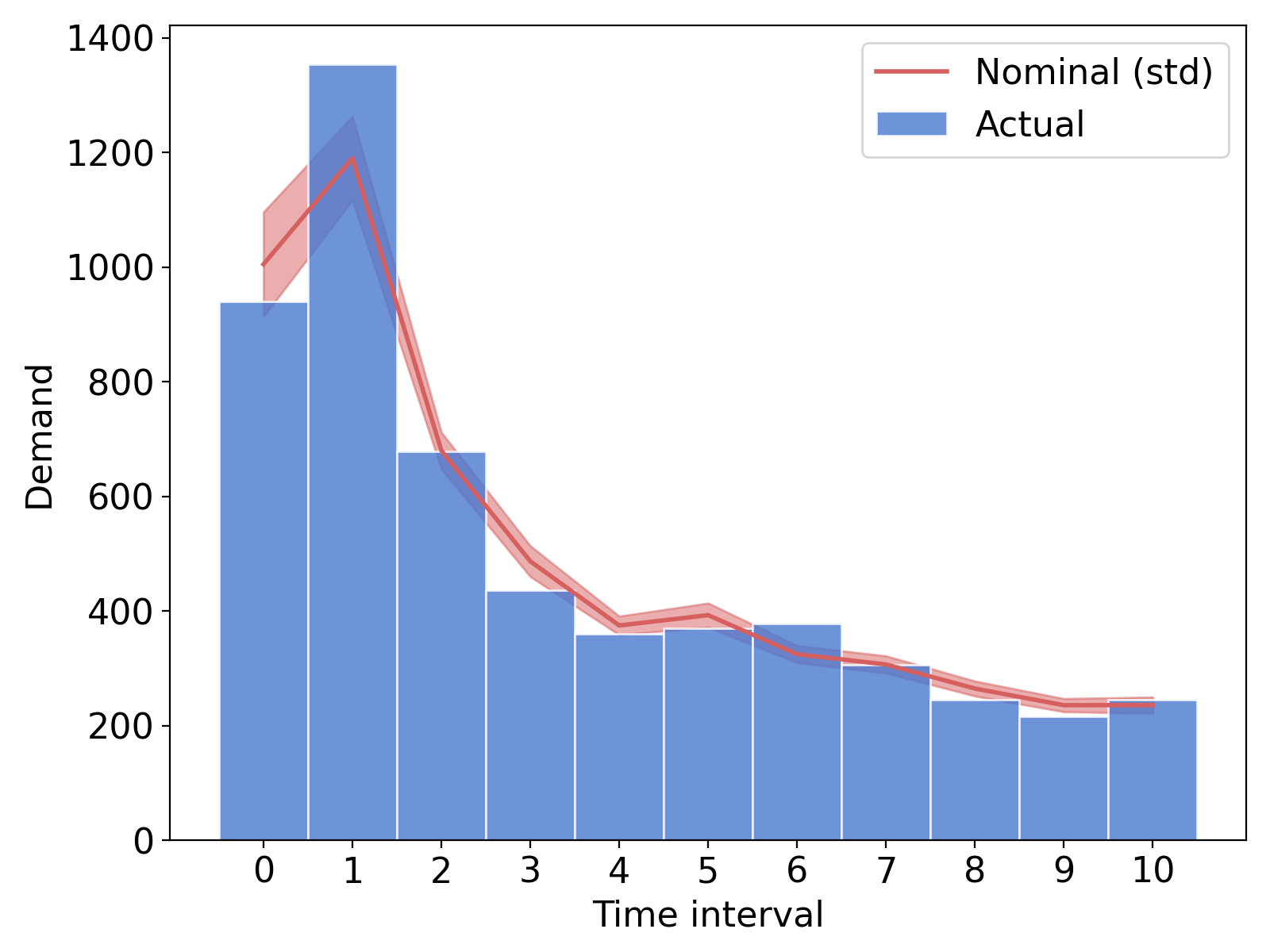}\label{fig_demand_total}}
\hfil
\subfloat[Demand comparison for each $(h,k)$]{\includegraphics[width=0.375\textwidth]{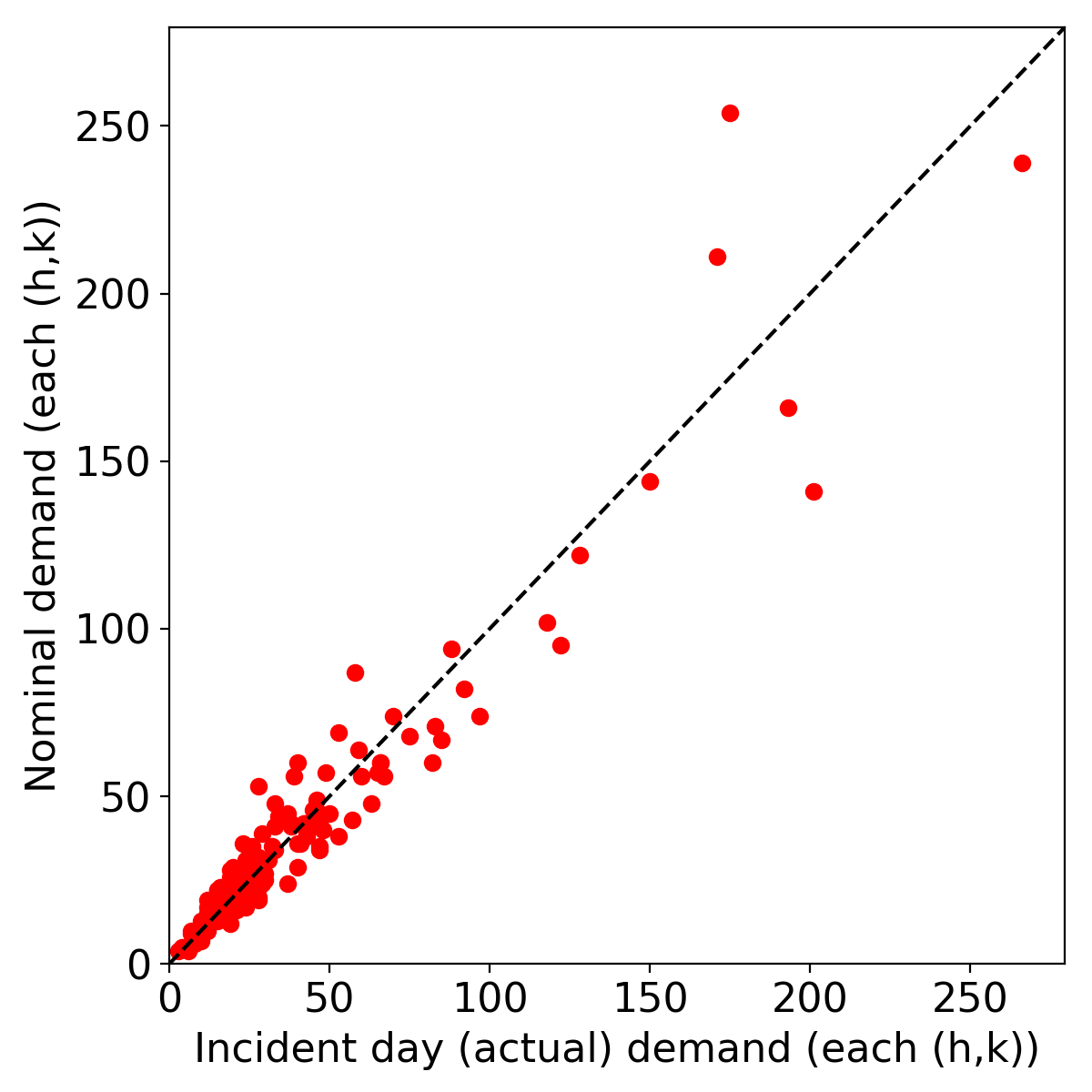}\label{fig_demand_hk}}
\caption{Demand patterns}
\label{fig_compare}
\end{figure}

Table \ref{tab_test} shows the results of the Mardia test of multivariate normality \citep{cain2017univariate} for demand samples. The Mardia test is used to check whether the sample's multivariate skewness and kurtosis are consistent with a multivariate normal distribution. If both are satisfied, we can assume the samples are multivariate normally distributed. We observe that, in Table \ref{tab_test}, the synthetic historical demands have consistent skewness but inconsistent kurtosis with the multivariate normal distribution, suggesting that they are not multivariate normally distributed. However, as skewness is a measure of the asymmetry of the probability distribution of a random variable about its mean, the Mardia Skewness testing shows that the demand distribution is symmetric. Hence, it is still reasonable to use the ellipsoidal uncertainty set to describe a symmetric distributed random variable. Moreover, as mentioned in Remark \ref{remark_1}, the distribution of a variable does not affect the definition of the uncertainty set (it only affects the calculation of probability guarantees).

\begin{table}[htbp]
\centering
\caption{Mardia test of multivariate normality}\label{tab_test}
\begin{tabular}{@{}llll@{}}
\toprule
Test            & p-value & Test            & p-value \\ \midrule
Mardia Skewness & 1.00    & Mardia Kurtosis & 0.00    \\ \bottomrule

\multicolumn{4}{p{8cm}}{\begin{tabular}[c]{@{}p{8.5 cm}@{}}\footnotesize{Note: The null hypothesis is that the samples are multivariate normally distributed. A small p-value indicates we are more likely to reject the null hypothesis.} \end{tabular}}

\end{tabular}
\end{table}

\subsection{Benchmark models}
The following approaches are used to obtain benchmark path shares.

\textbf{Uniform path shares}. The uniform path shares are defined as $p_{hkr} = \frac{1}{|R_k|} \; \forall \; r \in R_k$. This is a naive model corresponding to the intuition of ``distributing passengers to different paths'' when no information is available.   

\textbf{Capacity-based path shares}. The capacity-based path shares aim to assign passengers to different paths according to the path capacity. Specifically, for a path $r$ in OD pair $k$ and time $h$, we calculate the path capacity as the total available capacity of all vehicles passing through the first boarding station of the path (denoted as $C_{hkr}$). The capacity-based path shares are defined as
\begin{align}
    p_{hkr} = \frac{C_{hkr}}{\sum_{r \in R_k}C_{hkr}} \quad \forall \; r \in R_k, h \in \mathcal{H}, k \in \mathcal{K},
\end{align}
For example, for a path consisting of an NS bus route and the Green Line, $C_{hkr}$ is calculated as the total available capacity of all buses at the boarding station of the NS bus route during time interval $h$. The available capacity can be obtained from the simulation model using historical demand. The available capacity for the Blue Line (i.e., incident line) depends on the revised schedules during the incident (i.e., the service suspension is considered). When no trains operate on the Blue Line, the corresponding $C_{hkr}$ will be zero.

\textbf{Status-quo path shares}. The status-quo path shares are the inferred path choices of passengers on the incident day. During the incident period, the demand on the WE, NS, and parallel bus lines experienced an increase. The difference from the average demand on normal days can be seen as the number of passengers choosing the corresponding path. Hence, by identifying the demand increase for all nearby bus stops, we can get the number of passengers using the parallel bus, NS+Green, and WE+Brown paths for each OD pair $k$ and time interval $h$. However, the number of waiting passengers in the Blue Line cannot be directly inferred because the CTA system does not record the tap-out information. Hence, we approximate the proportion of waiting passengers based on survey results \citep{rahimi2019analysis}. \citet{rahimi2019analysis} used a survival model to analyze the waiting time tolerance of CTA riders during a service disruption. The model results provide the proportion of waiting passengers given different system recovery times. Therefore, the status-quo path shares are inferred as follows: 
\begin{itemize} 
    \item Step 1: Given the current time interval $h$ and the incident end time $T_e$, the remaining time until the end of the incident is $T_e - h$. Therefore, if passengers choose to wait, their waiting time will also be $T_e - h$. Based on the hazard model in \citet{rahimi2019analysis}, we can obtain the proportion of waiting passengers given the waiting time, denoted as $p_{\text{wait}}(T_e - h)$.
    \item Step 2: For each OD pair $k$ and time interval $h$, the number of passengers using the parallel bus, NS+Green, and WE+Brown paths can be calculated based on demand increase compared to the normal demand. Let the demand increase for path $r$ of OD pair $k$ at time $h$ be $DI_{hkr}$, where $r \in R_k\setminus\{r_\text{wait}\}$, $r_\text{wait}$ represents the path of waiting for the Blue Line.
    \item Step 3: The status quo path shares are calculated as follows:
    \begin{align}
        p_{hkr_\text{wait}} &= p_{\text{wait}}(T_e - h) \quad \forall \; h \in \mathcal{H}, k \in \mathcal{K},\\
        p_{hkr} &= (1-p_{hkr_\text{wait}})\cdot \frac{DI_{hkr}}{\sum_{r \in R_k\setminus\{r_\text{wait}\}} DI_{hkr}} \quad \forall \;r \in R_k\setminus\{r_\text{wait}\}, \; h \in \mathcal{H}, k \in \mathcal{K}
    \end{align}
\end{itemize}

\section{Results}\label{sec_results}
In this section, we demonstrate the model's performance in two steps. In the first step, results of the optimization model without uncertainty (i.e., the nominal model with $\rho_{1-\epsilon} = 0$) are compared with the three benchmark path shares. In the second step, we compare the results from the robust model with the results from the nominal model in order to assess the value of considering uncertainties in generating path recommendations.  

\subsection{Model convergence and computational time}\label{sec_model_converge}
Figure \ref{fig_convergence} shows the convergence of the nominal ($\rho_{1-\epsilon} = 0$) and robust (with $\rho_{1-\epsilon} = 0.84$) models. The simulation-based linearization and MSA successfully decrease the system travel time. The model converges within 35 iterations. Note that the optimal cost for the robust model is higher than the nominal model. This is expected since the robust model assumes the worst-case demand (by definition with higher system travel time). The performance of the corresponding path recommendations will be evaluated based on the actual demand (discussed in the next section).

\begin{figure}[htb]
\centering
\includegraphics[width = 0.6 \linewidth]{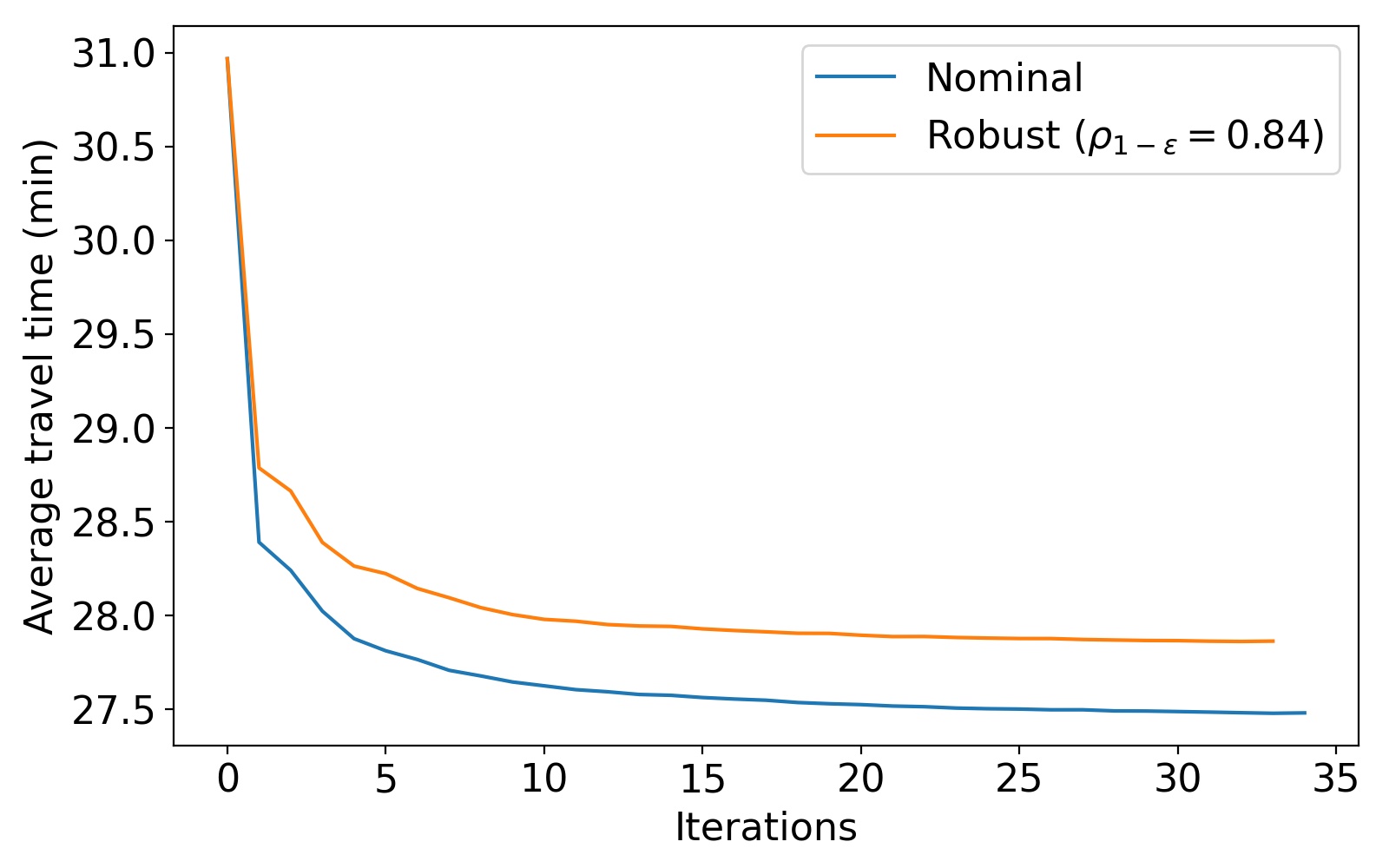}
\caption{Convergence of optimization models}
\label{fig_convergence}
\end{figure}

The mode's computational time mainly depends on the speed of the simulation and the number of iterations. Solving Eqs. \ref{eq_ro_all} and \ref{eq_WD} is quite efficient because of the tractable SOCP formulations. Currently, running one simulation for the CTA system for 2 hours takes around 0.6 minutes. Hence, the total solving time for a RO model (assuming 35 iterations) is around 23 minutes. 

The computational time of 23 minutes is too long for real-world rolling horizon implementation. However, since the bottleneck is the simulation process, future studies can improve the model's efficiency by enhancing the speed of the simulation model (e.g., code with C++). If the simulation time is reduced to 5 seconds, the total solution time will be reduced to around 3 minutes, which enables the real-world rolling horizon implementation with 15 minutes intervals (including times for updating inputs). Another way to simplify the model is to reduce the look ahead horizons. Now we consider a 2-hour horizon until the end of the incident. With rolling horizon implementation, since the incident information will be updated in real-time, we may only need to look ahead for 1 hour and wait for new information to come. Reducing the look-ahead time can significantly simplify the model and reduce the computational time.

\subsection{Model evaluation}
The optimization model only utilizes information about the nominal demand and the associated uncertainty set. The actual demand is unknown when running the model (otherwise there are no uncertainties). After obtaining the path shares (either from optimization or the benchmark models), the recommendation strategies are evaluated based on the actual incident day demand using the simulation model. We assume passengers would follow the path recommendation. The simulation model can output the travel times of every passenger in the system, and can be used to compare the performance of the path shares obtained from the various approaches. Performance is measured in terms of average travel time and average waiting time.

\subsection{Nominal vs. Benchmark models}
% may also add the travel time without incident. 

% \subsubsection{Average travel time comparison}
Table \ref{tab_result} compares the results for different path shares, The result of the no incident scenario is also shown for comparison. The average travel times are calculated over all passengers (a total of 27,007 passengers) and the passengers who originally planned to use the Blue Line (i.e., passengers who are provided with recommendations, a total of 5,531 passengers, a subset of the 27,007 passengers). Results show that the optimization-based path shares outperform all benchmark models. For all passengers in the system, the average travel time is reduced by 9.1\% compared to the status quo. And for the incident line passengers, the reduction is even higher (20.6\%). 

Recommendations based on the uniform path shares result in worse performance than the status quo scenario. This implies that current passengers' choices are not random and show some rationality. The capacity-based path shares can also reduce the system travel time significantly (by 6.9\%). However, as the capacity-based path recommendations do not capture the spatial and temporal changes in available capacity due to passenger flow redistribution, they are worse than the optimization-based results. A more comprehensive discussion on the performance comparison between the optimization model and the capacity-based model can be found in \ref{sec_discuss_cap}.

Compared to the no-incident scenario, we find that the influence of incidents is significant. Path recommendations can only alleviate the impact of service disruption but are far from eliminating. Even with the optimization-based path recommendations, we still have more than two times of travel time for incident-line passengers compared to the no-incident situation. 

% \begin{table}[htbp]
% \centering
% \caption{Average travel time comparison}\label{tab_result}
% \resizebox{\textwidth}{!}{
% {\renewcommand{\arraystretch}{1} % line space
% \begin{tabular}{@{}lcccc@{}}
% \toprule
% Scenarios            & \multicolumn{2}{c}{All passengers (\# 27,003)} & \multicolumn{2}{c}{Incident-line passengers (\# 5,531)} \\ \cmidrule(l){2-5} 
%                       & Avg TT$^1$ (min)$^2$        &  Avg WT (min)       & Avg TT$^1$ (min)            & Avg WT (min)           \\ \midrule
% No incident                & 21.81                        & 2.11          & 18.95                            & 2.09               \\\hline
% Uniform                & 31.02 (+1.7\%)                       &  18.45         & 54.64                            & +6.4\%               \\
% Status quo             & 30.49 (0.0\%)                        & 20.11               & 51.34                            & 0\%                     \\
% Capacity-based         & 28.36 (-6.9\% )                          & 9.30         & 43.23                            & -15.8\%              \\
% \textbf{Optimization (nominal)} & \textbf{27.71 (-9.1\%)}        &  10.58         & \textbf{40.75 (-20.6\%) }    &  9.34            \\ \bottomrule

% \multicolumn{4}{l}{\begin{tabular}[c]{@{}l@{}}$^1$ TT: travel time; WT: waiting time\\
% $^2$: Numbers in the parentheses indicate the changes compared to the status quo scenario \end{tabular}}

% \end{tabular}
% }
% }
% \end{table}
\begin{table}[htbp]
\centering
\caption{Average travel time comparison}\label{tab_result}
\resizebox{\textwidth}{!}{
{\renewcommand{\arraystretch}{1} % line space
\begin{tabular}{@{}lcccc@{}}
\toprule
Scenarios            & \multicolumn{2}{c}{All passengers (\# 27,003)} & \multicolumn{2}{c}{Incident-line passengers (\# 5,531)} \\ \cmidrule(l){2-5} 
                       & Avg travel time (min)        & \% change$^1$       & Avg travel time (min)            & \% change$^1$            \\ \midrule
No incident                & 21.81                        & -          & 18.95                            & -               \\\hline
Uniform                & 31.02                        & +1.7\%          & 54.64                            & +6.4\%               \\
Status quo             & 30.49                        & 0\%               & 51.34                            & 0\%                     \\
Capacity-based         & 28.36                        & -6.9\%          & 43.23                            & -15.8\%              \\
\textbf{Optimization (nominal)} & 27.71                        & \textbf{-9.1\%}          & 40.75                            & \textbf{-20.6\%}              \\ \bottomrule

\multicolumn{4}{l}{\begin{tabular}[c]{@{}l@{}}$^1$: changes compared to the status quo scenario \end{tabular}}

\end{tabular}
}
}
\end{table}

Figure \ref{fig_compare_tt} shows the average travel time and waiting time for different paths for all incident line passengers. We observe that the optimization-based path recommendations have more consistent travel time across the four types of paths, implying a better utilization of the system's capacity. However, for other recommendation strategies, passengers using parallel buses have significantly longer travel times than those using other alternatives. Figure \ref{fig_compare_tt} also shows that the average waiting time for the status quo scenario is around 30 minutes, which means most passengers chose to use the parallel bus during the incident, causing severe congestion. However, with the optimization-based path shares, the average waiting time for the parallel bus is less than 5 minutes (around a headway). 

\begin{figure}[H]
\centering
\subfloat[Average travel time for different paths]{\includegraphics[width=0.5\textwidth]{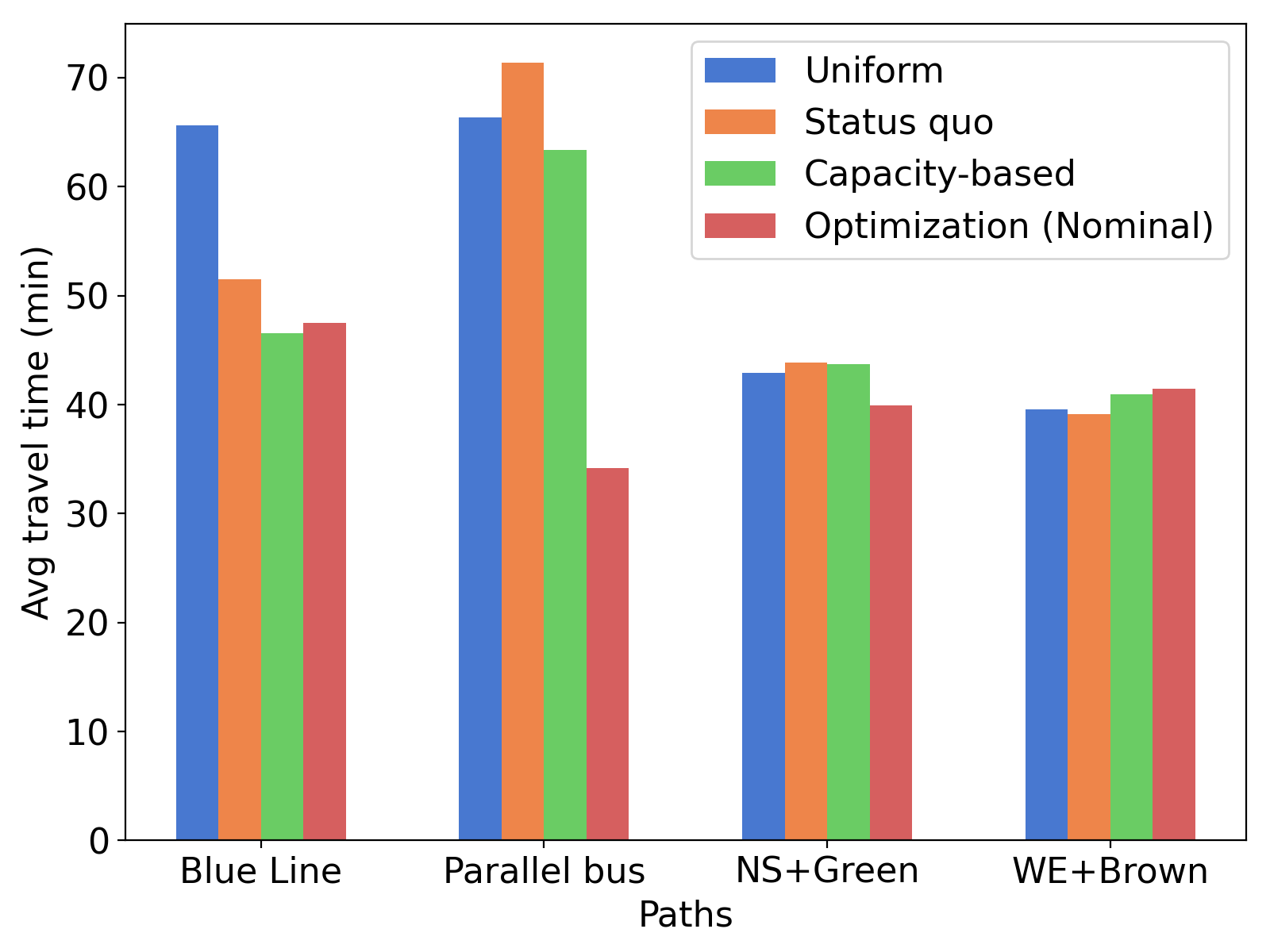}\label{fig_travel_time}}
\hfil
\subfloat[Average waiting time for different paths]{\includegraphics[width=0.5\textwidth]{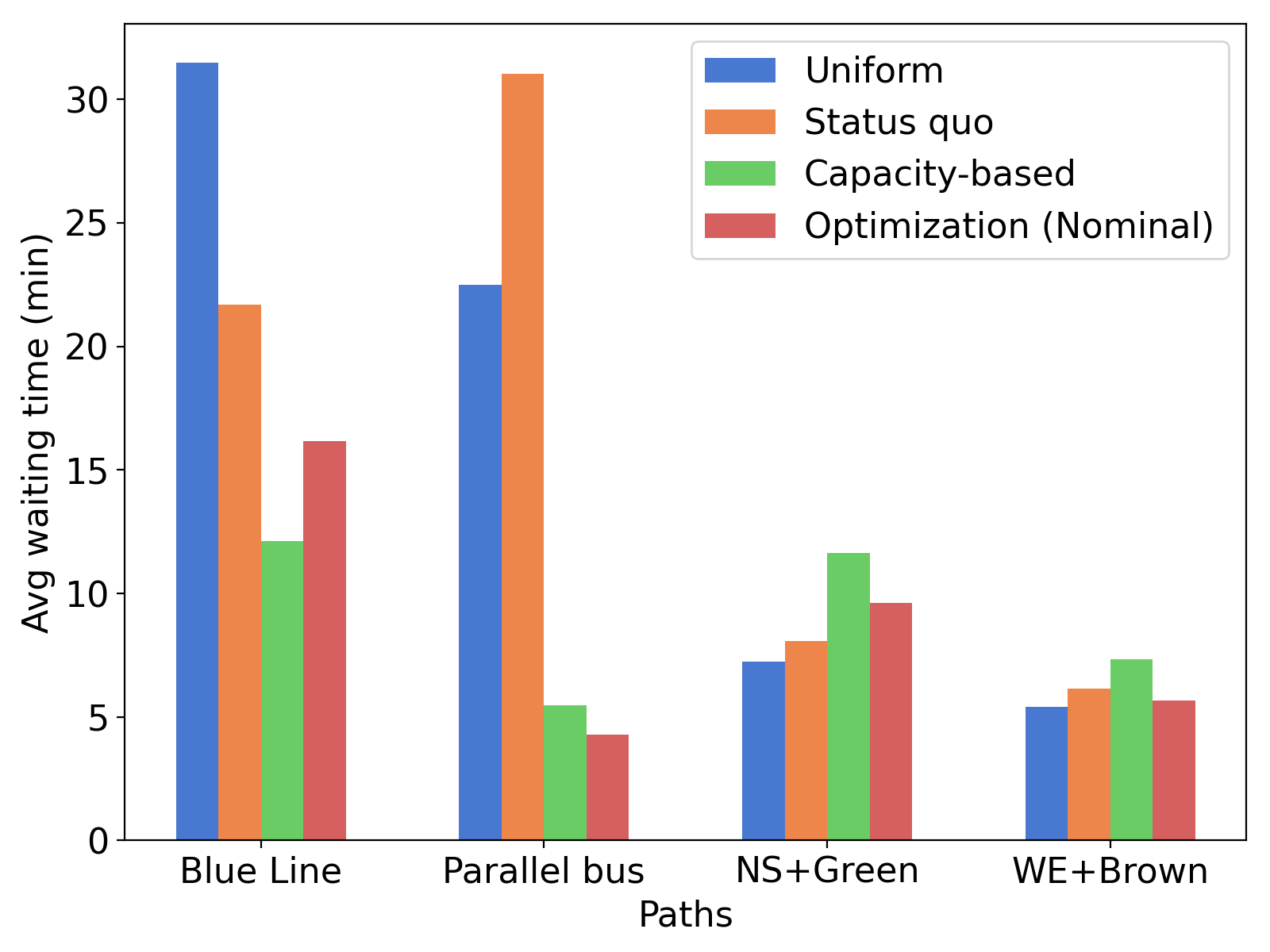}\label{fig_waiting}}
\caption{Comparison of average travel time and waiting time of different paths for incident line passengers}
\label{fig_compare_tt}
\end{figure}

The objective of this study is to minimize the system travel time. However, under the optimal path shares, some passengers' travel time may be increased compared to the status quo. Figure \ref{fig_tt_changes} shows the distribution of changes in individual travel time (optimization-based minus the status quo) for all passengers whose path choice under the recommendation scenario is different than their choice in the status quo scenario. Most passengers experience lower travel times. However, some passengers become worse off after following the path recommendations. This is a typical drawback of system optimal (first-best) assignment \citep{lawphongpanich2010solving}. Future studies may explore a Pareto-improving (second-best) path recommendation that ensures no individual becomes worse-off. In reality, when implementing the recommendations, some paths that lead to extremely worse travel time compared to the status quo can be dropped from the solution. 
%However, in reality, after recommending the paths, passengers do not know the counterfactual (i.e., the travel time without path recommendation) 

\begin{figure}[htb]
\centering
\includegraphics[width = 0.9 \linewidth]{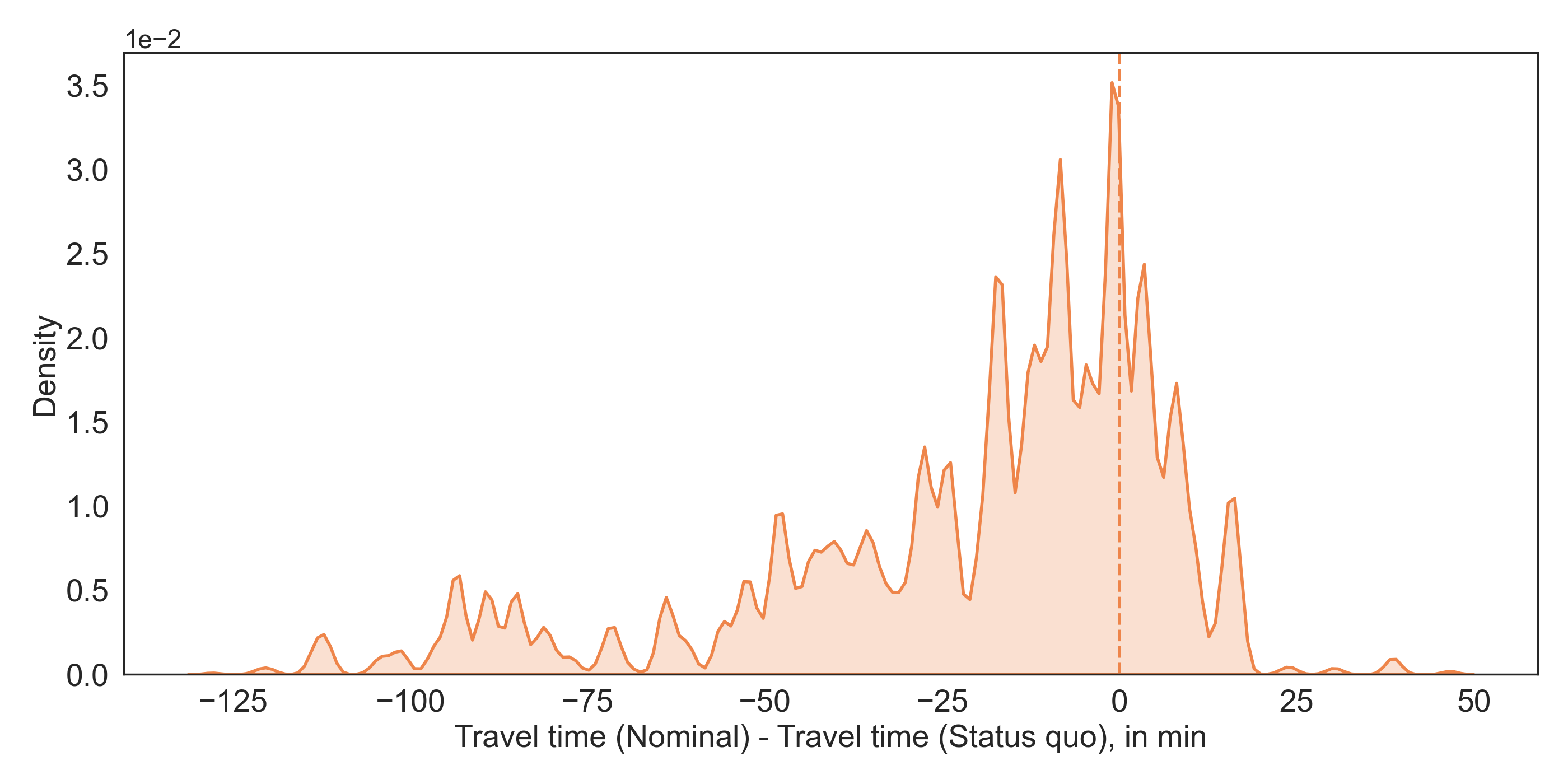}
\caption{Distribution of the change in individual travel time (not including passengers without changes as they will distort the distribution with too much density concentrated at zero)}
\label{fig_tt_changes}
\end{figure}

\subsection{Robust models vs. Nominal model}
\subsubsection{Model comparison under actual demand}\label{sec_model_ad}
Figure \ref{fig_RO} compares the results, in terms of travel time, of the RO approach with different values of $\rho_{1-\epsilon}$ under the actual demand. For all values of the robust model except for $\rho_{1-\epsilon} = 2.33$, the RO approach shows better performance than the nominal model. This implies that considering the demand uncertainty in determining the recommendation can further improve the effectiveness of path recommendation strategies. The best value is $\rho_{1-\epsilon} = 0.84$, where the travel time for the incident line passengers is reduced by 2.91\% compared to the nominal model. Note that the percentage decreases are relatively small because some passengers' travel times are not changed. If we only look at incident-line passengers with travel time changes, the average travel times are 47.6 min and 37.9 min for the nominal and RO ($\rho_{1-\epsilon} = 0.84$) scenarios, respectively, where the travel time reductions are 20.4\%. 

\begin{figure}[H]
\centering
\includegraphics[width = 0.6 \linewidth]{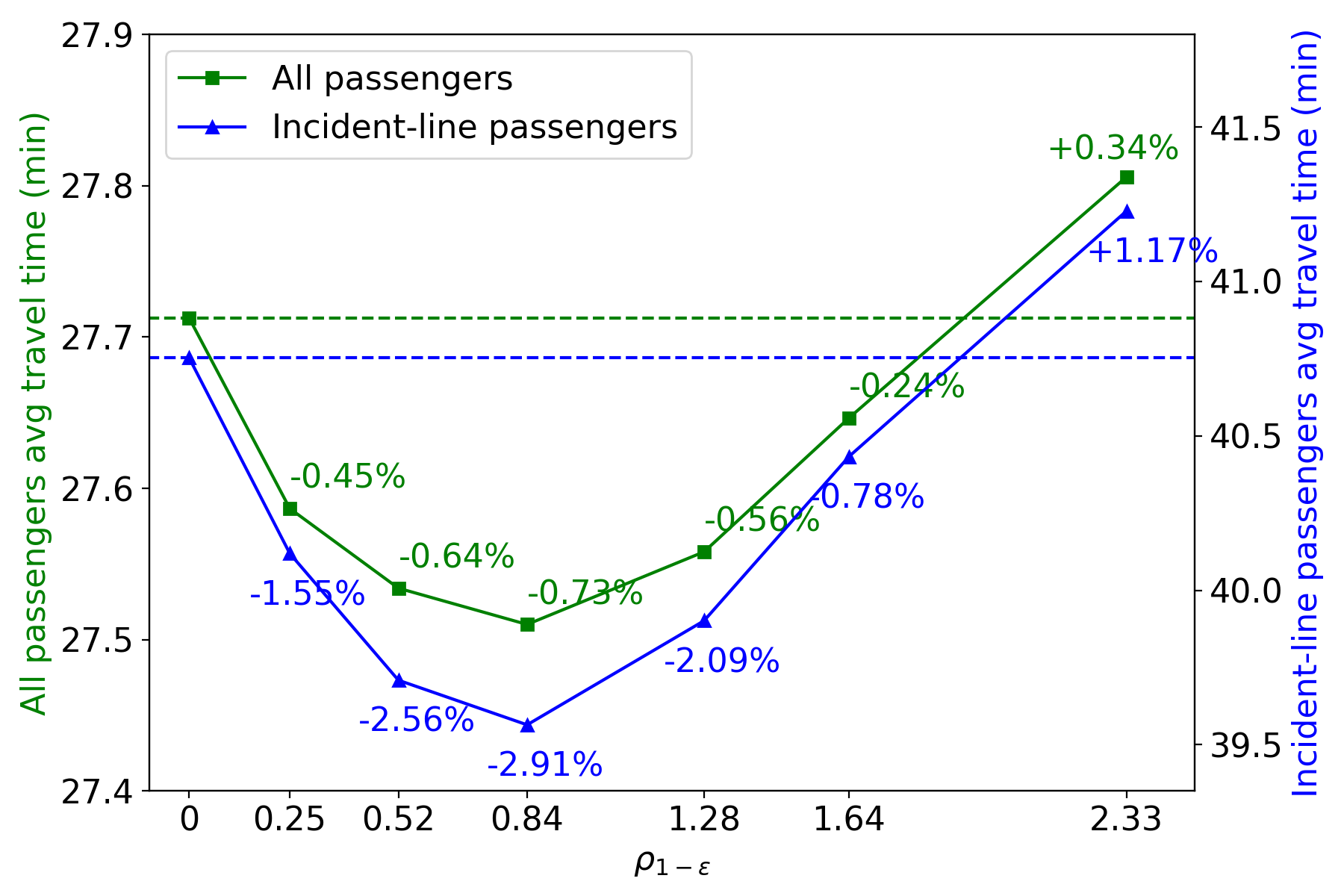}
\caption{Performance of RO. The percentage changes are compared to the nominal scenario}
\label{fig_RO}
\end{figure}

Note that using $\rho_{1-\epsilon} = 2.33$ results in the largest uncertainty set compared to other values. This reflects a very conservative scenario where the agency prefers to plan against a very high realization of demand. In this case, the worst-case demand patterns may deviate from the actual demand too much, thus performing worse than the nominal model. Figure \ref{fig_wd} illustrates the worst-case demand for different values of $\rho_{1-\epsilon}$. The worst-case demands for the $\rho_{1-\epsilon} = 0.52, 0.84, 1.28$ scenarios are closer to the actual demand, while $\rho_{1-\epsilon} = 2.33$ overestimates the demands, especially for the earliest periods ($h=0, 1$) (which are the most critical periods). These results are consistent with the travel time performance in Figure \ref{fig_RO}.

\begin{figure}[H]
\centering
\includegraphics[width = 0.6 \linewidth]{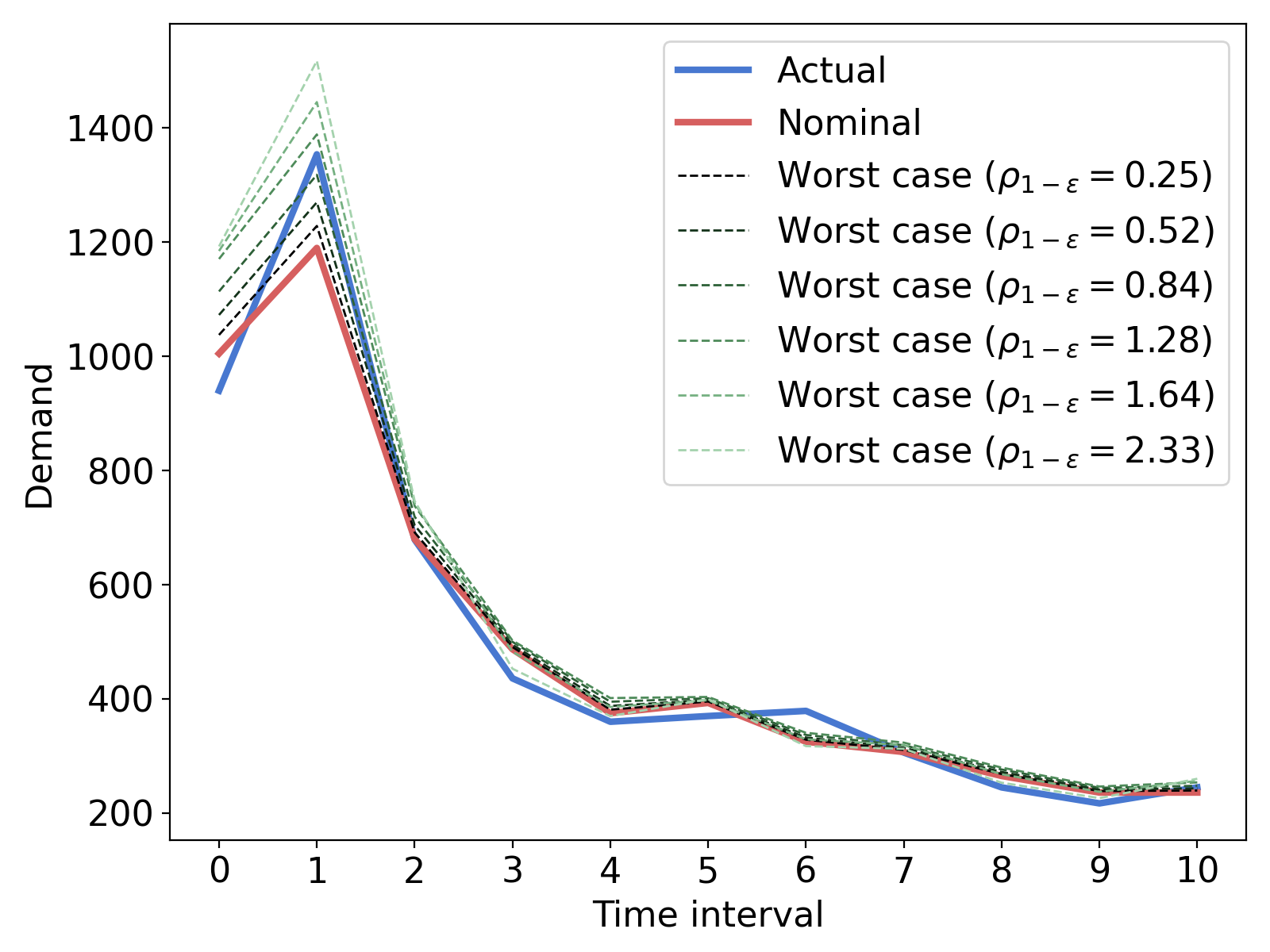}
\caption{Worst-case demand patterns}
\label{fig_wd}
\end{figure}

\subsubsection{Model comparison under random demand}\label{sec_model_rd}
To further validate the model's performance, we test the performance of the solution obtained from the RO approach on the 16 demand samples generated in Section \ref{sec_quant_uncertainty_set}.  These demand samples represent different possible realizations of the incident day demand. Figure \ref{fig_rd} shows the compassion of the random demand samples versus the actual and nominal demands. Notice that the random demand samples include both high and low demand scenarios, which can better validate the performance of the RO approach under different demand patterns. 

\begin{figure}[H]
\centering
\includegraphics[width = 0.6 \linewidth]{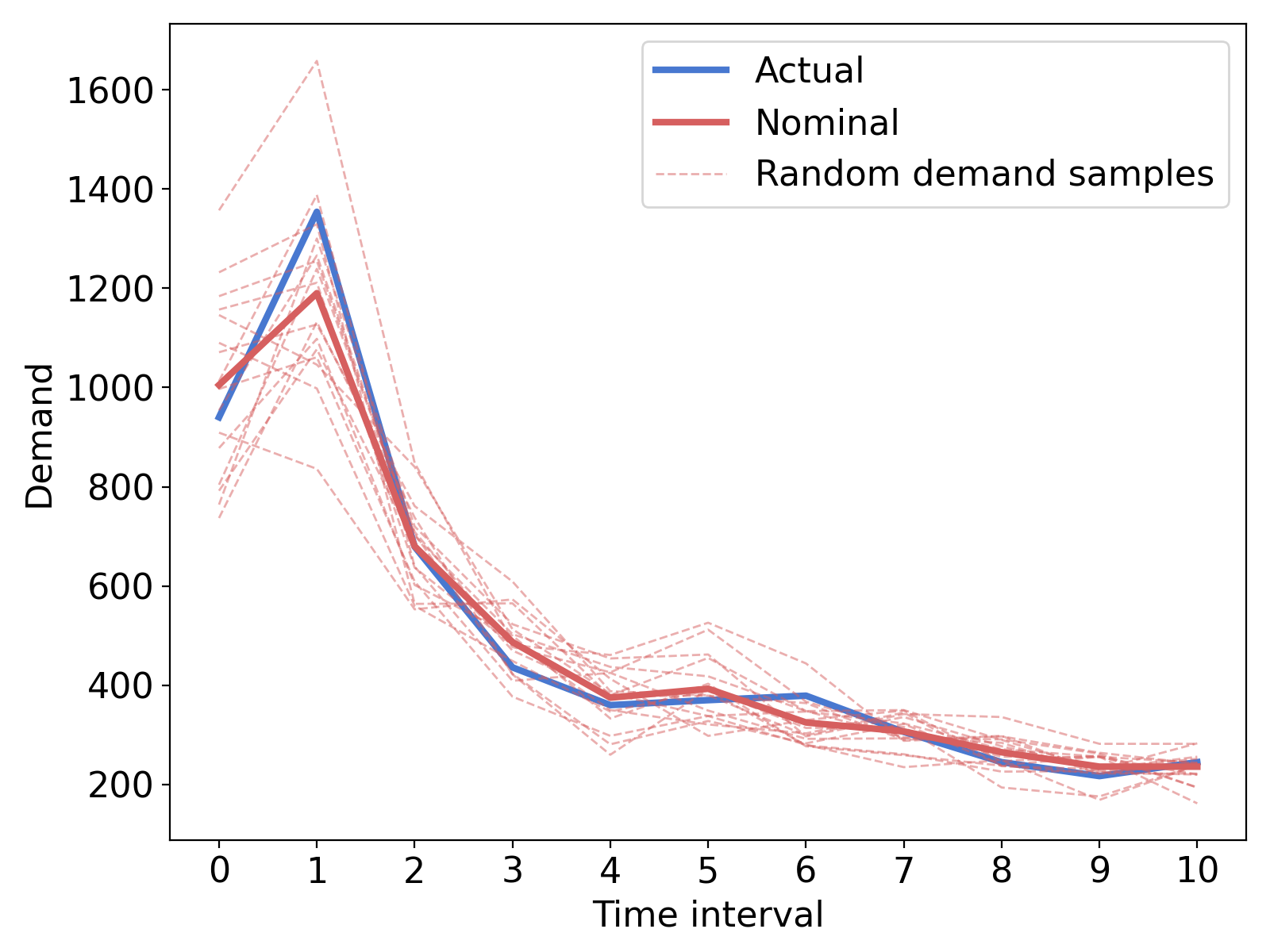}
\caption{Random demand patterns for experiments}
\label{fig_rd}
\end{figure}

Table \ref{tab_result_RO} compares the results of average travel time for different RO models. The numbers in the table are the mean values of the 16 experiments. The performances are similar to the results under the actual demand.  The RO approach shows better performance than the nominal model for all values of $\rho_{1-\epsilon}$ the robust model except for $\rho_{1-\epsilon} = 2.33$. The reasons may be that the RO approach focuses more on critical OD pairs and time intervals where the path recommendations for them are considered more important for system performance. 
% It is worth noting that
\begin{table}[htbp]
\centering
\caption{Average travel time comparison for RO models}\label{tab_result_RO}
%\resizebox{\textwidth}{!}{
{\renewcommand{\arraystretch}{1} % line space
\begin{tabular}{@{}lcccc@{}}
\toprule
Models            & \multicolumn{2}{c}{All passengers} & \multicolumn{2}{c}{Incident-line passengers} \\ \cmidrule(l){2-5} 
                       & Avg travel time (min)        & \% change$^1$       & Avg travel time (min)            & \% change$^1$            \\ \midrule
Nominal ($\rho_{1-\epsilon} = 0$)               &  27.79                       & -          & 41.08                           & -               \\
$\rho_{1-\epsilon} = 0.25$               & 27.70                        & -0.32\%          & 40.57                            & -1.23\%               \\
$\rho_{1-\epsilon} = 0.52$            & 27.65                        & -0.48\%               & 40.24                            & -2.05\%                     \\
$\rho_{1-\epsilon} = 0.84$         & 27.64                        & -0.54\%          & 40.13                  & -2.31\%              \\
$\rho_{1-\epsilon} = 1.28$ & 27.68                        & {-0.39\%}          & 40.41                            & {-1.62\%}              \\ 
$\rho_{1-\epsilon} = 1.64$ & 27.74                        & {-0.17\%}          & 40.83                            & {-0.60\%}              \\ 
$\rho_{1-\epsilon} = 2.33$ & 27.86                        & {+0.27\%}          & 41.47                            & {+0.96\%}   \\          
\bottomrule
\multicolumn{4}{l}{\begin{tabular}[c]{@{}l@{}}$^1$: changes compared to the nominal model \end{tabular}}

\end{tabular}
%}
}
\end{table}

\subsubsection{Model comparison under worst-case demand}
Theoretically, the RO model should show the best performance under the worst-case demand. Table \ref{tab_result_WD_compare} shows the comparison between the nominal model and RO models under the worst-case demand (the worst-case demand patterns are shown in Figure \ref{fig_wd}). For each value of $\rho_{1-\epsilon}$, we obtain the results for both the nominal model and the robust model by assuming the actual demand is the worst-case demand. Note that with a higher value of $\rho_{1-\epsilon}$, the worst-case demand patterns become worse and the average travel time of passengers will increase regardless of which model we use. 

As shown in Table \ref{tab_result_WD_compare}, if the ``actual demand'' (i.e., the demand for model evaluation) is the worst-case demand, the robust models consistently outperform the nominal models. And the improvement is higher than that of the previous experiments. Moreover, the higher the value of $\rho_{1-\epsilon}$ (i.e., more extreme demand patterns), the higher the improvement of robust models compared to the nominal models. This emphasizes the importance of considering demand uncertainties under extreme demand patterns. 

\begin{table}[htbp]
\centering
\caption{Average travel time comparison under the worst-case demands}\label{tab_result_WD_compare}
\resizebox{\textwidth}{!}{
{\renewcommand{\arraystretch}{1} % line space
\begin{tabular}{@{}cccccc@{}}
\toprule
\multirow{2}{*}{Uncertainty set}    & \multicolumn{1}{l}{\multirow{2}{*}{Total demand}} & \multicolumn{2}{c}{All passengers avg travel time (min)} & \multicolumn{2}{c}{Incident-line passengers avg travel time (min)} \\ \cmidrule(l){3-6} 
                           & \multicolumn{1}{l}{}                              & Nominal model         & Robust model (\% change$^1$)         & Nominal model              & Robust model (\% change)              \\ \midrule
$\rho_{1-\epsilon} = 0.25$ & 27,084                                             & 27.99                 & 27.67 (-1.14\%)                                & 41.77                      & 40.43 (-3.21\%)                                   \\
$\rho_{1-\epsilon} = 0.52$ & 27,195                                             & 28.12                 & 27.75 (-1.32\%)                          & 42.14                      & 40.62 (-3.61\%)                                       \\
$\rho_{1-\epsilon} = 0.84$ & 27,344                                             & 28.26                 & 27.86 (-1.41\%)                                & 42.47                      & 40.89 (-3.72\%)                                   \\
$\rho_{1-\epsilon} = 1.28$ & 27,525                                             & 28.37                 & 27.91 (-1.62\%)                                 & 42.99                      & 41.22 (-4.12\%)                                  \\
$\rho_{1-\epsilon} = 1.64$ & 27,522                                             & 28.61                 & 28.05 (-1.96\%)                               & 43.81                      & 41.74 (-4.72\%)                                    \\
$\rho_{1-\epsilon} = 2.33$ & 27,520                                             & 28.93                 & 28.28 (-2.25\%)                                 & 45.28                      & 42.65 (-5.80\%)                                   \\ \bottomrule
\multicolumn{4}{l}{\begin{tabular}[c]{@{}l@{}}$^1$: changes compared to the nominal model with the same worst-case demand \end{tabular}}

\end{tabular}
}
}
\end{table}

\section{Conclusion and discussion}\label{sec_conclusion}
In this paper, we propose a path recommendation model to mitigate congestion during public transit disruptions. Passengers with different ODs and departure times are recommended alternative paths to use such that the total system travel time is minimized. To tackle the non-analytical formulation of travel times due to left behind, we propose a simulation-based first-order approximation to transform the original problem into a linear program and solve the new problem iteratively with MSA. Uncertainties in demand are modeled using RO techniques to protect the path recommendation strategies against inaccurate estimates. A real-world rail disruption scenario in the CTA system is used as a case study. Results show that even without considering uncertainty, the nominal model can reduce the system travel time by 9.1\% (compared to the status quo), and outperforms the benchmark capacity-based path recommendation. The average travel time of passengers in the incident line is reduced more (-20.6\% compared to the status quo). After incorporating the demand uncertainty, the robust model further reduces the system travel time. The best robust model with $\rho_{1-\epsilon} = 0.84$ decreases the average travel time of incident-line passengers by 2.91\% compared to the nominal model. 

The performance improvement by incorporating demand uncertainty is not very significant. The reason may be that demand variations at the incident situation have a limited impact on the optimal path shares. Notice that the demand during an incident is already very high for the system (due to the reduced supply level). Hence, the path recommendation patterns under nominal and worst-case demand may be similar. However, the methodology presented in this study provides a general way to deal with PT demand uncertainty. It can be used for other operations control, optimization, planning, or recommendation applications.

Though we discussed potential model extensions with rolling horizon and incident duration uncertainty, we did not implement these extensions in the case study as the focus has been on the methodology for solving the problem. Incorporating real-time information as an adaptive RO would generally increase model performance \citep{bertsimas2011theory}. This presents an interesting future research direction. Other future research directions include the following. 1) Current demand uncertainty sets need to be quantified with a budget factor $\rho_{1 - \varepsilon}$. The choice of budget factor usually relies on numerical testing \citep{bertsimas2012adaptive, guo2021robust}. Future studies may also develop data-driven uncertainty quantification methods to automate the hyperparameter tuning task. 2) As shown in Figure \ref{fig_tt_changes}, the system optimal path recommendation may result in worse-off travel time for some passengers, causing equity and fairness issues. Future studies may consider incorporating Pareto-improving constraints to ensure that all passengers are better-off if following our recommendation. 3) In this study, we assume that passengers follow the recommendation. Non-compliance, however, if present, may lead to the actual path flows deviating from the optimal ones. Future research may focus on approaches for path recommendations that capture behavior uncertainty. 4) Finally, this study presents an OD-based (aggregated) path recommendation regime. Passengers with the same OD and departure time are treated homogeneously. In reality, different passengers may have different preferences on path choices. And these preferences can affect their compliance with recommendations. Future studies can develop an individualized path recommendation system considering heterogeneous passenger preferences.

\section{Authors’ contribution}
\textbf{Baichuan Mo}: Conceptualization, Methodology, Software, Formal analysis, Data Curation, Writing - Original Draft, Writing - Review \& Editing, Visualization.  \textbf{Haris N. Koutsopoulos}: Conceptualization, Supervision, Formal analysis, Writing - Review \& Editing, \textbf{Zuo-Jun Max Shen}: Conceptualization, Supervision. \textbf{Jinhua Zhao:} Conceptualization, Supervision, Project administration, Funding acquisition.

\section{Acknowledgement}
The authors would like to thank Chicago Transit Authority (CTA) for their support and data availability for this research. 

\appendix
\appendixpage
\section{Capturing supply changes by adjusting timetable}\label{sec_sim_supply_change}
Since all the operations can be described by timetable, it is reasonable to capture the supply changes in disruptions by adjusting the timetable. In this appendix, we show that our simulation model is able to capture the “partially blocked tracks” at a platform with complex configurations (e.g., different train capacities or different lines). That is, even if only one specific track failed in a platform, this type of disruption can be captured by the change in vehicles' timetables. 

For example, consider a platform with 2 different lines A and B (Figure \ref{fig_multi_platform}). Suppose that only the track associated with Line A is blocked. And operators decide to let Lines A and B share the same remaining track. This operation change can be captured by adjusting the train’s timetable for Lines A and B (i.e., trains may have higher headway, and they cannot use the platform simultaneously in the new timetable). In this example, the second vehicle of Line A is delayed to 7:30 and the second vehicle of Line B is delayed to 7:40. 

\begin{figure}[H]
\centering
\includegraphics[width = 0.7 \linewidth]{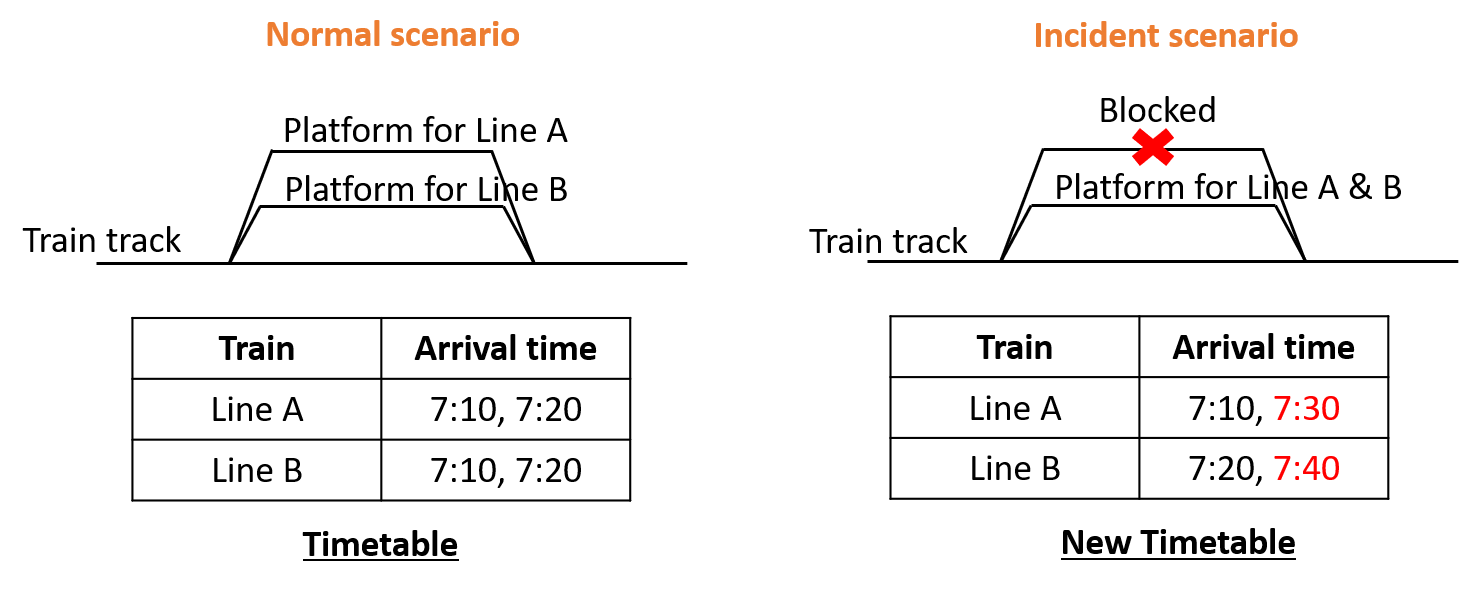}
\caption{Example of representing supply changes as timetable changes in a multi-platform scenario}
\label{fig_multi_platform}
\end{figure}

\section{Simulation-based first-order approximation}\label{sec_sim_first_order}

\subsection{Calculation of $T_{hkr}^{\text{Q}}(\boldsymbol{\tilde{f}})$}\label{sec_sim_TQ}
Let $\mathcal{V}_{hkr}^{b}$ be the set of vehicles that the $\mathcal{M}_{hkr}$ passengers board at station $b$. Adding an additional passenger to $\mathcal{M}_{hkr}$ means one more passenger boards one of the vehicles in $\mathcal{V}_{hkr}^{b}$. Let $\mathbbm{1}_{\{\text{Full}_v^b\}}$ be an indicator of whether vehicle $v$ is full or not after its departure from station $b$. Then the total increase in system travel time for passengers queuing behind $\mathcal{M}_{hkr}$ is:
\begin{align}
T_{hkr}^{\text{Q}}(\boldsymbol{\tilde{f}}) =\sum_{b\in\mathcal{B}_{hkr}} \sum_{v \in \mathcal{V}_{hkr}^{b}}\frac{\mathbbm{1}_{\{\text{Full}_v^b\}} \cdot W_v^b}{|\mathcal{V}_{hkr}^{b}|}
\end{align}
where $\mathcal{B}_{hkr}$ is the set of all boarding stations for $\mathcal{M}_{hkr}$ passengers (in the example of Figure \ref{fig_marginal_cost}, $\mathcal{B}_{hkr} = \{a_1, a_5\}$). $W_v^b$ is the headway of vehicle $v$ at station $b$. The sum over all vehicles is because we do not specify the exact vehicle that the additional passenger will board, and thus take the average over all vehicles. In the example of Figure \ref{fig_marginal_cost}, since there are two boarding stations for $\mathcal{M}_{hkr}$ ($a_1, a_5$), $T_{hkr}^{\text{Q}}(\boldsymbol{\tilde{f}})$ is approximately two headways if the vehicles are full.

\subsection{Calculation of $T_{hkr}^{\text{O}}(\boldsymbol{\tilde{f}})$}\label{sec_sim_TO}
Let $\mathcal{O}_{hkr}^v$ be the set of all on-board stations for $\mathcal{M}_{hkr}$ and vehicle $v \in \mathcal{V}_{hkr}^{b}$. For example, for vehicles in Line 1 in Figure \ref{fig_marginal_cost}, $\mathcal{O}_{hkr}^v$ will be $a_2$, $a_3$, and $a_4$. Then the travel time increase for passengers waiting at on-board stations is:
\begin{align}
T_{hkr}^{\text{O}}(\boldsymbol{\tilde{f}}) =\sum_{b\in\mathcal{B}_{hkr}} \sum_{v \in \mathcal{V}_{hkr}^{b}}\frac{1}{{|\mathcal{V}_{hkr}^{b}|}}\sum_{a \in \mathcal{O}_{hkr}^{v}} {\mathbbm{1}_{\{\text{Full}_v^a\}} \cdot W_v^a}
\end{align} 

\section{Path-passenger matching}\label{sec_path_pass}
After obtaining the optimal path shares $p_{hkr}^*$, the operator may need to know which path to provide to each such that the final path shares are close to  $p_{hkr}^*$, especially when passengers have different preferences and may not follow the unpreferred recommendations. In this section, we define a path-passenger matching problem as a solution for this challenge.

Consider a passenger $j$ with a path set $\mathcal{R}_i$. His/her inherent preference (utility) of using path $r\in \mathcal{R}_i$ is denoted as $V_{i}^{r}$. If path $r'$ was recommended, the impact of the recommendation on the utility of path $r$ is denoted as $I_{j,r'}^r$. Hence, his/her overall utility of using path $r$ can be represented as
\begin{align}
    U_{j}^r = V_{j}^{r} + \sum_{r' \in \mathcal{R}_j}x_{i,r'} \cdot I_{j,r'}^r + \xi_{j}^{r} \quad \forall r \in \mathcal{R}_i, j \in \mathcal{J}
\end{align}
where $\xi_{j}^{r}$ is the random error. $\mathcal{J}$ is the set of all passengers that need recommendations. $x_{j,r'} = 1$ if passenger $j$ is recommended path $r'$, otherwise $x_{j,r'} = 0$. Let $ \pi_{i,r'}^r$ be the conditional probability that passenger $j$ chooses path $r$ given that the recommended path is $r'$. Assuming a utility-maximizing behavior, we have
\begin{align}
    \pi_{j,r'}^r = \mathbb{P}( V_{j}^{r} + I_{j,r'}^r + \xi_{j}^{r} \geq V_{j}^{r''} +  I_{j,r'}^{r''} + \xi_{j}^{r''}, \; \forall r'' \in \mathcal{R}_j) 
\end{align}
Different assumptions for the distribution of $\xi_{j}^{r}$ can lead to different expressions. For example, if $\xi_{p}^{r}$ are i.i.d. Gumbel distributed, the choice probability reduces to multinomial logit model \citep{train2009discrete} and we have 
\begin{align}
    \pi_{j,r'}^r = \frac{\exp(V_{j}^{r} +  I_{j,r'}^r) }{\sum_{r''\in \mathcal{R}_j}  \exp(V_{j}^{r''} + I_{j,r'}^{r''} )}
    \label{eq_pi_prob}
\end{align}
The value of $ V_{j}^{r}$ and $I_{j,r'}^r$ can be calibrated using data from an individual-level survey or smart card, which deserves separate research. For those without such information, this information can be approximated by the population average. When developing the path-passenger matching formulation, we assume $\pi_{j,r'}^r$ is known. Figure \ref{fig_bh_un} shows an example for the conditional probability matrix. The specific values assume that paths with recommendations are more likely to be chosen.  
\begin{figure}[H]
\centering
\includegraphics[width = 0.6 \linewidth]{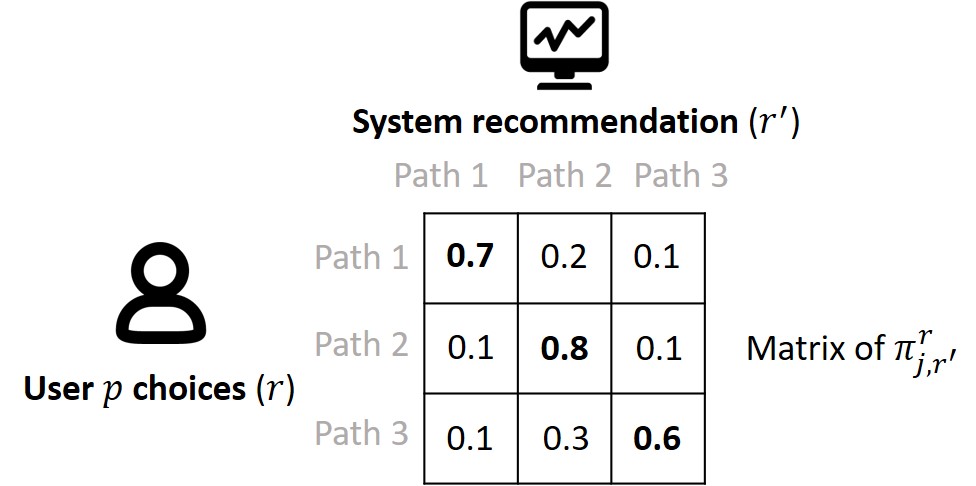}
\caption{Example of conditional path choice probability}
\label{fig_bh_un}
\end{figure}
The conditional probability $\pi_{j,r'}^r$ captures the individual's inherent preference for different paths as well as the response to the recommendation system.

The expected path flow for $(h,k,r)$ is 
\begin{align}
{\mu}_{hkr}(\boldsymbol{x}) =  \sum_{j\in \mathcal{J}_{hk}} \sum_{r' \in \mathcal{R}_k}x_{j,r'} \cdot \pi_{j,r'}^r + q_{hkr} \label{eq_mean_flow} \quad \forall (h,k,r)\in \mathcal{F}
\end{align}
where $\boldsymbol{x} =: (x_{j,r})_{j\in\mathcal{J},r\in\mathcal{R}_j}$. $q_{hkr}$ (constant) is the flow of passengers in $(h,k,r)$ that do not need recommendations. $ \mathcal{J}_{hk} \subseteq \mathcal{J}$ is the set of passengers  with OD pair $k$ and departure time $h$ who need recommendations. 

Suppose the value of $\pi_{p,r'}^r$ is known, we can formulate the path-passenger matching problem as an integer linear program:
\begin{subequations}
\label{eq_ipr}
\begin{align}
    \quad \min_{\boldsymbol{x}} \quad & \sum_{(h,k,r) \in \mathcal{F}} |{\mu}_{hkr}(\boldsymbol{x}) - d_{hk}\cdot p_{hkr}^*| \\
    \text{s.t.} \quad
    & \text{Constraint } (\ref{eq_mean_flow})  \\
    & d_{hk} = \sum_{r\in\mathcal{R}_{k}} \mu_{hkr}(\boldsymbol{x}) \quad \forall h\in\mathcal{H},k
    \in \mathcal{K}\\
     & \sum_{r\in R_j} x_{j,r} = 1 \quad \forall j \in \mathcal{J} \\
    & x_{j,r} \in \{0,1\} \quad  \forall j \in \mathcal{J}, r\in \mathcal{R}_j \label{const_IRR2}
\end{align}
\end{subequations}
The objective function aims to minimize the difference between the expected path flow and the optimal path flow. Solving Eq. \ref{eq_ipr} yields which path should be recommended to each passenger. It is worth noting that one could also solve the path-passenger matching problem and the path recommendation problem simultaneously, which is equivalent to an individual-based path recommendation problem \citep{mo2023individual}.

\section{Validation of the simulation model}\label{sec_validate_sim}
Usually, a transit simulator is validated by ``OD exit flow'' (i.e., the number of tap-out passengers for a specific OD pair and time interval). This is because we only input the ``OD entry flow'' (i.e., the number of tap-in passengers for a specific OD pair and time interval). Since the simulator will output the tap-out time for each passenger, comparing the model-output OD exit flow with the ground truth (obtained by AFC data with tap-out information) provides validation for the model. 

However, in this study, the CTA system does not have tap-out information because it is an open system, implying that the ground truth OD exit flow is not available. But CTA adopted a destination system called ``ODX''. The ODX algorithm is developed by \citet{sanchez2017inference}. It is shorthand for ‘‘origin, destination, and transfer inference algorithm,’’ an extension of the O-D inference algorithm proposed by \citet{zhao2007estimating}. It takes automatically collected data, including AVL and automatic fare collection (AFC), as inputs and infers both destinations and transfers in a tap-on-only transit system, including locations and times. Given a series of tap-in records for a given smart card ID, the tap-out information is inferred as follows: 1) if the current tap-in time is close to the previous tap-in time, the current vehicle ``stage'' is part of a transfer journey from the previous stage and the alighting location of the previous stage is inferred as the closest stop on that route to the boarding location of the current (second) stage; 2) if there is a large time gap between the current tap-on and the previous tap-int, the alighting location of the previous journey is inferred as the closest stop on the previous route to the boarding location of the current journey assuming passengers’ travel patterns are symmetrical and the distance between the inferred alighting location and the subsequent boarding location meets maximum distance criteria. More details can be found in \citet{zhao2007estimating} and \citet{sanchez2017inference}.

We can treat OD exit flow output by ODX as the ground truth. The comparison is based on the data on a normal weekday without incident. We also aggregated the flows by destinations for better visualization. Figure \ref{fig_od_exit_flow} shows the comparison between OD exit flow between 9:00-10:00 AM at the top 10 stations in the analysis area (see Figure \ref{fig_incident}). The flows between simulation and ground truth are consistent, implying that the simulation can well capture the passenger and vehicle dynamics. A more comprehensive validation of the simulation model can be found in \citet{mo2020capacity}, where the case study is based on the Hong Kong Mass Transit Railway system.  

\begin{figure}[H]
\centering
\includegraphics[width = 0.6 \linewidth]{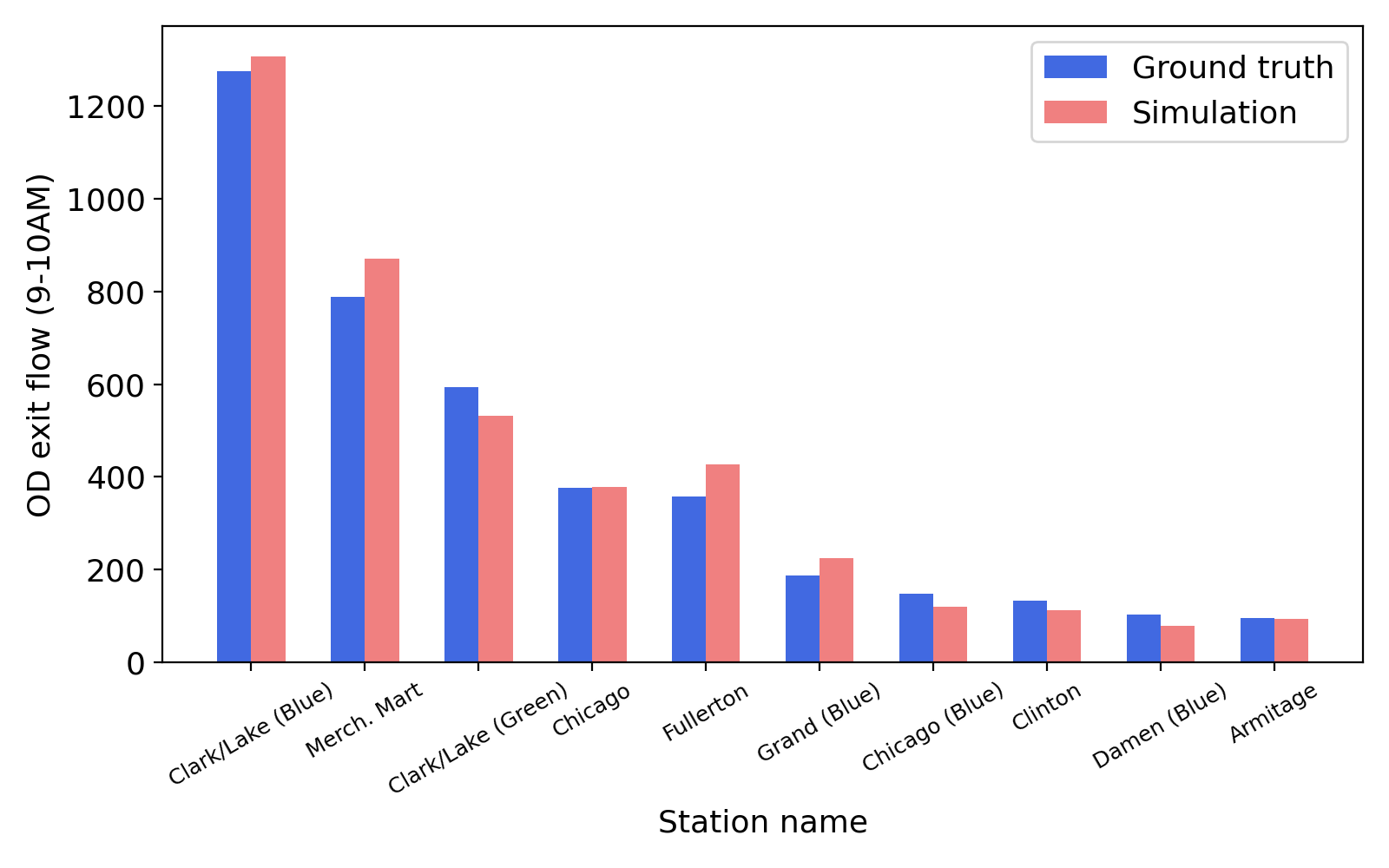}
\caption{Comparison between simulation of ground truth OD exit flows}
\label{fig_od_exit_flow}
\end{figure}

\section{Discussions on the model performance compared to the capacity-based path recommendations}\label{sec_discuss_cap}
The improvement of the proposed method compared to the capacity-based method mainly comes from the ``network-level'' optimization for solving the system optimal flows. The proposed model considers the dynamics between upstream and downstream decisions, which are ignored in the capacity-based path recommendations. For example, in the capacity-based path recommendation, when calculating the available capacity for downstream passengers, the possibly occupied capacity by the recommended upstream passengers is not captured. This is because capacity-based path recommendation is a simple heuristic and does not consider the interaction between the recommendations at differentiation stations from the spatial and temporal aspects. However, the optimization-based methods (both nominal and RO models) are able to capture these interactions through network-level optimization. 

In our case study, the improvement compared to the capacity-based method is not significant (around 5\% for incident line passengers, see Table \ref{tab_result}). The reason is that there are not many upstream and downstream interactions in the Blue Line case study (see Figure \ref{fig_incident}). Only the parallel bus line has this problem (i.e., the recommendation from upstream passengers may occupy the capacity of downstream passengers). The NS and WE buses have independent lines to connect to the Green and Brown Lines, respectively. And Green and Brown lines have enough capacity to serve passengers from the Blue Line. Therefore, we did not see significant improvement. 

This can be further evidenced by the path shares comparison (Figure \ref{fig_path_share_comp}). The path shares shown in the figure are the weighted average over all OD pairs and time intervals with weights equal to the corresponding demand. We observe that the optimization model de-prioritizes the use of parallel buses. However, the capacity-based model, since cannot capture the upstream and downstream recommendation interactions, over-recommends passengers to the parallel buses. It is worth noting that, when there are no upstream and downstream interactions, the capacity-based path recommendation can be very close to the “system-optimal” path shares (if the travel times of alternative paths are similar). 

\begin{figure}[H]
\centering
\includegraphics[width = 0.6 \linewidth]{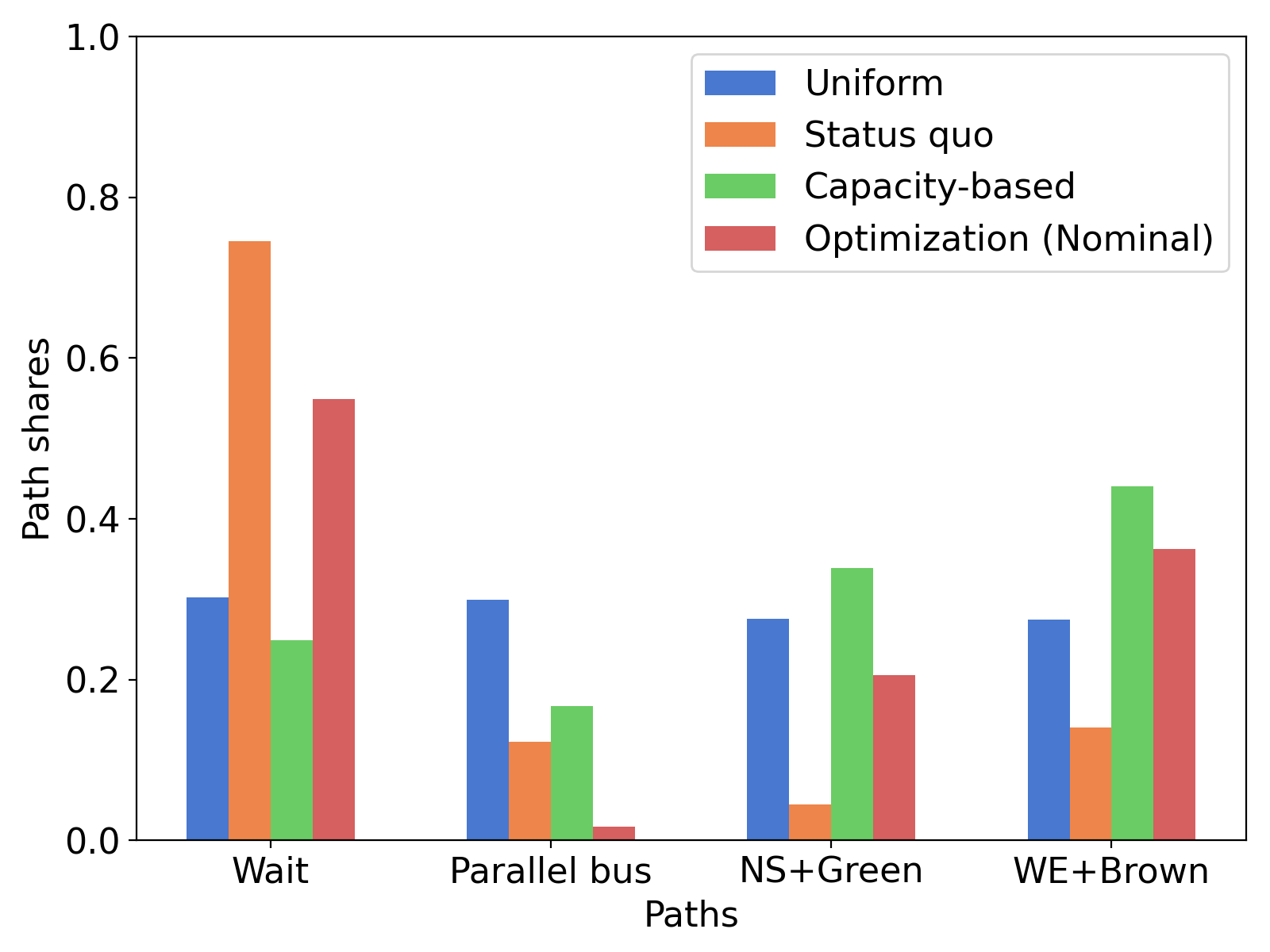}
\caption{Path shares comparison for different recommendation strategies}
\label{fig_path_share_comp}
\end{figure}

\bibliography{mybibfile}

\end{document}